\documentclass{article}
\usepackage{graphicx} 
\usepackage{amsmath,amsthm,amssymb,mathtools}
\usepackage{tikz}
\usepackage{hyperref}
\usepackage[all]{xy}
\usepackage{subcaption} 
\usepackage{tabularx,array}
\usepackage{cancel}
\usepackage{authblk}
\newtheorem{theorem}{Theorem}[section]
\newtheorem{remark}[theorem]{Remark}
\newtheorem{proposition}[theorem]{Proposition}
\newtheorem{corollary}[theorem]{Corollary}

\newtheorem{definition}[theorem]{Definition}
\newtheorem{lemma}[theorem]{Lemma}

\hypersetup{hypertexnames=false}

\DeclareMathOperator{\con}{\mathsf{C}}
\DeclareMathOperator{\notcon}{\cancel{ \con}}
\DeclareMathOperator{\Overl}{\mathsf{O}}
\DeclareMathOperator{\Q}{\mathsf{Q}}
\DeclareMathOperator{\G}{\mathsf{G}}
\DeclareMathOperator{\notG}{\cancel{ \G}}
\DeclareMathOperator{\R}{\mathsf{R}}
\DeclareMathOperator{\notR}{\cancel{ \R}}

\newcommand{\KO}[1]{\mathrm{K}_{\Overl}(#1)}
\newcommand{\KC}[1]{\mathrm{K}_\T(#1)}

\newcommand{\cn}{\mathrm{c}}
\newcommand{\cbf}{\mathbf{c}}
\newcommand{\one}{\mathbf{1}}
\newcommand{\zero}{\mathbf{0}}
\newcommand{\A}{\mathbf{A}}

\newcommand{\D}{\mathbf{D}}
\newcommand{\X}{\mathbf{X}}
\newcommand{\T}{\mathcal{T}}
\newcommand{\U}{\mathcal{U}}

\newcommand{\boxn}{\blacksquare}
\newcommand{\diamon}{\blacklozenge}
\title{A Modal Expansion of Kleene Algebras via Twist Structures}
\author[1,2]{Sergio Celani\thanks{\texttt{scelani@exa.unicen.edu.ar}}}
\author[1]{Paula Mench\'on\thanks{\texttt{mpmenchon@nucompa.exa.unicen.edu.ar}}}
\author[1,2]{Valeria Miguelez\thanks{\texttt{mmiguelez@alumnos.exa.unicen.edu.ar}}}

\affil[1]{NUCOMPA, Departamento de Matem\'atica, Facultad de Ciencias Exactas, Universidad Nacional del Centro de la Provincia de Buenos Aires, 7000 Tandil, Argentina}
\affil[2]{CONICET}

\date{}

\begin{document}

\maketitle

\begin{abstract}
Kleene triples, introduced by Jalali, provide a representation of arbitrary Kleene algebras in terms of twist-products. We introduce modal Kleene algebras and modal Kleene triples, and establish a categorical equivalence between the corresponding categories, extending Jalali's duality to the modal setting. We study the centered case and show that Kleene algebras with implication can be naturally described within this framework by viewing implication as a family of modal operators. Finally, we develop a topological duality for modal Kleene triples.
\end{abstract}




\section{Introduction}

Kleene algebras arise from the normal distributive $i$-lattices introduced by J. A. Kalman \cite{kalman1958lattices}. They are distributive lattices equipped with an involution satisfying the normality condition
\(x \wedge \sim x \leq y \vee \sim y\).
A fundamental construction due to Kalman associates with every bounded distributive lattice \(L\) a centered Kleene algebra \(K(L)\), built on pairs \((a,b)\in L\times L\) satisfying \(a\wedge b=0\), with involution \(\sim(a,b)=(b,a)\). Cignoli \cite{cignoli1986class} showed that this construction defines a functor from bounded distributive lattices to centered Kleene algebras, and established a categorical equivalence when the Kleene side is restricted to algebras satisfying the interpolation property, later shown to be equivalent to condition~(CK) of \cite{2017kleene}. In particular, this yields an equivalence between Heyting algebras and centered Nelson algebras \cite{cignoli1986class,sendlewski1990nelson}, making it possible to transfer results from the theory of Heyting algebras to Nelson algebras, the algebraic semantics of Nelson's constructive logic with strong negation \cite{nelson1949constructible,Rasiowa58}. Since Kalman's construction always produces centered algebras, the representation of arbitrary Kleene algebras requires a different approach. This was provided by Jalali \cite{jalali, viglizzo99algebras} via the notion
of \emph{Kleene triples} $\langle L, R, F \rangle$, where $L$ is a bounded
distributive lattice, $R$ is an independence relation on $L$, and $F$ is a
Boolean filter. Independence relations are dual to the contact relations
studied in \cite{duntsch2006topological, duntsch2008distributive}. Jalali
proved that the category of Kleene algebras is isomorphic to the category of
Kleene triples, with Kalman's construction corresponding to the special case
where $R$ is the complement of the contact relation \textit{Overlap}.

In this paper we extend the Kleene triple framework by incorporating modal
operators. We introduce modal Kleene algebras and modal Kleene triples and
establish a categorical equivalence between them, generalizing Jalali's
theorem. One of the main features of our approach is that, on the lattice side, we require only a few conditions for the modal operators. Nevertheless, the corresponding class of modal Kleene algebras forms a variety. We study two particular cases of this equivalence. The first
concerns centered modal Kleene algebras; the centered case has played a
central role throughout the history of the subject, as it was the first
setting in which equivalences of this kind were established
\cite{cignoli1986class, jalali}. The second concerns
Kleene algebras with implication. Motivated by the observation that
for each $a \in L$ the map $a \to (-)$ can be regarded as a modal operator
on $L$, so that an implication on $L$ amounts to a parametric family of
modal operators, we establish a categorical
equivalence between lattices with implication and Kleene algebras with
implication via the modal Kleene triple framework. As an application, we
show that subresiduated Nelson algebras \cite{noemi2025}, which extend
Nelson algebras by replacing the Heyting implication with a subresiduated
one, can be represented in this setting. In \cite{noemi2025} an equivalence
is established between subresiduated lattices and centered subresiduated
Nelson algebras satisfying~(CK); here we obtain a representation for the
non-centered case.

The paper is organized as follows. In Section~\ref{preliminaries} we recall the relevant
background on Jalali's representation and the centered case. Section~\ref{sec:modal}
introduces modal Kleene algebras and modal Kleene triples and establishes
the main categorical equivalence. Subsection~\ref{subsec:centered} analyzes the centered case
in detail. Section~\ref{sec:implicative} develops the representation for Kleene algebras
with implication. Finally, in Section~\ref{sec:topological-duality} we develop the topological duality for
the algebraic structures introduced in Section~\ref{sec:modal}, providing a complete dual
characterization of modal Kleene triples.

\section{Preliminaries}\label{preliminaries}

A \emph{De Morgan algebra} is an algebra $\A=\langle A,\wedge,\vee,\sim,0,1\rangle$ of type $(2,2,1,0,0)$ such that $\langle A,\wedge,\vee,0,1\rangle$ is a bounded distributive lattice and the following equations hold for all $x,y\in A$:
\begin{align}
\quad x &= \sim \sim x, \tag{Inv}\label{Inv}\\
\quad \sim (x \vee y) &= \sim x \wedge \sim y. \tag{DM}\label{DM}
\end{align}

A \emph{Kleene algebra} is a De Morgan algebra in which the following inequality holds for all $x,y \in A$:
\[
x \wedge \sim x \leq y \vee \sim y.
\]
An element $z\in A$ is called a \emph{center} of a Kleene algebra $\A$ if $z=\sim z$. It is well known that such an element, whenever it exists, is unique.

A \emph{centered Kleene algebra} $\A=\langle A,\wedge,\vee,\sim,\cn,0,1\rangle$ is an expansion of a Kleene algebra by a constant symbol $\cn$, interpreted as its center.

Let $\D=\langle D,\wedge,\vee,0,1\rangle$ be a bounded distributive lattice. Then, as shown in \cite{kalman1958lattices}, the following set
\[
\KO{D} := \{(a,b) \in D \times D : a \wedge b = 0\}
\]
equipped with the operations defined by
\begin{equation}
\label{eq:twist-operations}
\begin{array}{ll}
(1) \ (a,b) \cap (e,d) := (a \wedge e, b \vee d); & (2) \ (a,b) \cup (e,d) := (a \vee e, b \wedge d); \\
(3) \ \sim(a,b) := (b,a); & (4) \ \zero := (0,1); \\
(5) \ \one := (1,0); & (6) \ \cbf := (0,0),
\end{array}
\end{equation}
forms a centered Kleene algebra $\mathbf{K}_{\Overl}(\D)=\langle \KO{D},\cap,\cup,\sim,\cbf,\zero,\one\rangle$. 

A \emph{contact algebra} is an expansion of a Boolean algebra equipped with a binary relation $\con$, known as the \emph{contact relation}, which satisfies a set of axioms. This algebraic structure serves as one of the simplest formal counterparts of \emph{mereotopology}, a field concerned with the study of spatial relationships and regions.
A well-studied generalization of contact algebras (see \cite{duntsch2006topological,duntsch2008distributive}) drops the requirement of complementation, thereby weakening the underlying structure from a Boolean algebra to a distributive lattice. This generalization preserves the essential features of the original theory while extending its applicability to a broader class of structures.

Let $\D$ be a bounded distributive lattice. A binary relation $\con$ on $D$ is called a \emph{contact relation} \cite{duntsch2006topological,duntsch2008distributive} if it satisfies the following conditions:
\begin{align}
 &(\forall a)\; 0\notcon a; \label{C0}\tag{C0} \\
 &(\forall a)\; [a \neq 0 \Rightarrow a\con a]; \label{C1}\tag{C1} \\
 &(\forall a)(\forall b)\; [a\con b \Rightarrow b\con a]; \label{C2}\tag{C2} \\
 &(\forall a)(\forall b)(\forall c)\; [a\con b \text{ and } b \leq c \Rightarrow a\con c]; \label{C3}\tag{C3} \\
 &(\forall a)(\forall b)(\forall c)\; [a\con (b \vee c) \Rightarrow a\con b \text{ or } a\con c]. \label{C4}\tag{C4}
\end{align}

A pair $\langle \D, \con \rangle$ is called a \emph{distributive contact lattice} (DCL) if $\D$ is a bounded distributive lattice and $\con$ is a contact relation on $D$.

\begin{remark}\label{rem:overlap}
    Let $\D$ be a bounded distributive lattice. The binary relation $\Overl$ defined by
    \[
    a \Overl b \quad \text{iff} \quad a \wedge b \neq 0
    \]
    is a contact relation known as the \emph{overlap relation} \cite{duntsch2008distributive}. It is the smallest contact relation on $D$, i.e., for any contact relation $\con$ on $D$, if $a \Overl b$ then $a \con b$. In particular, it follows that for any contact relation, if $a \notcon b$, then $a \wedge b = 0$. 
\end{remark}

Kalman's construction can be reformulated in terms of contact as follows:
\[
\KO{D} = \{(a, b) \in D \times D : a \cancel{\Overl} b \},
\]
where the contact relation is the overlap. This naturally leads to the question of how the construction behaves when an arbitrary contact relation is considered in place of the overlap relation. This issue was addressed by Jalali in \cite{jalali}, who established an equivalence between contact lattices and Kleene algebras via twist structures. In particular, from the definition provided by Jalali, it can be shown that independence relations are the complements of the contact relations defined above. The results we present below follow this approach.

A filter $F$ of a distributive lattice $\D$ is said to be \emph{Boolean} if the quotient $D/F$ is a Boolean algebra or if $F=D$.

\begin{definition}
    Let $ \D$ be a distributive lattice, $ \con$ be a contact relation on $D$, and $ F$ be a Boolean filter of $ \D$. The triple $ \T=\langle \D, \con, F \rangle$ is called a \emph{Kleene triple} if for every $ a \in D$ there exists $ b \in D$ such that $ a\notcon b$ and $ a \vee b \in F$.
\end{definition}

The category of Kleene triples, denoted by $ \mathbf{Kt}$, is the category whose objects are Kleene triples and whose morphisms from
$\T_1=\langle\D_1,\con_1,F_1\rangle$ to $\T_2=\langle\D_2,\con_2,F_2\rangle$ are bounded lattice homomorphisms $f\colon \D_1\to \D_2$ such that $ f[F_1] \subseteq F_2$ and, $f(a)\con_2 f(b)$ implies $a\con_1 b$.

We denote by $\mathbf{Kl}$ the category of Kleene algebras and their homomorphisms.

The following proposition shows that starting from a Kleene triple $\T = \langle \D, \con, F \rangle$, one can construct a Kleene algebra $\mathbf{K}(\T)$.

\begin{proposition}
\label{prop:twist-kleene}
Let $\T = \langle \D, \con, F \rangle$ be a Kleene triple. Then the structure
\[
\mathbf{K}(\T) = \langle \KC{D}, \cap, \cup, \sim, \zero, \one \rangle
\]
is a Kleene algebra, where the universe is
\[
\KC{D} := \{ (a, b) \in D \times D : a\notcon b \text{ and } a \vee b \in F \},
\]
and the operations are as defined in \ref{eq:twist-operations}.
\end{proposition}

We also refer to $\mathbf{K}(\T)$ as the \emph{twist structure} associated with $\T$.

\begin{proposition}
\label{prop:Kfunctor-mor}
Let $\T_1=\langle \D_1,\con_1,F_1\rangle$ and $\T_2=\langle \D_2,\con_2,F_2\rangle$ be in $ \mathbf{Kt}$ and let $f \colon \T_1 \to  \T_2$ be a morphism. Then the map $\mathrm{K}(f) \colon \mathbf{K}(\T_1) \to \mathbf{K}(\T_2)$ defined by
\[
    \mathrm{K}(f)((a, b)) := (f(a), f(b))
\]
is a morphism of Kleene algebras. Consequently, the assignments
\[
\T\mapsto\mathbf{K}(\T),
\qquad
f\mapsto\mathrm{K}(f),
\]
define a functor $\mathrm{K}:\mathbf{Kt}\to\mathbf{Kl}$. 
\end{proposition}

Let $\A$ be a Kleene algebra. The set of \emph{negative elements} of $A$ is defined as:
\[
A^{-}:= \{ x \in A : x \leq \sim x \}.
\]
Dually, we define the set of \emph{positive elements} as:
\[
A^{+}  := \{ x \in A : \sim x \leq x \}.
\]

It is easy to see that $A^{-}= \{ x \wedge \sim x : x \in A \}$ and $A^{+}= \{ x \vee \sim x : x \in A \}$. It is also easy to prove that $A^{-}$ is an ideal and dually $A^{+}$ is a filter of $\A$.

Given $x \in A$, we denote by $|x|$ the equivalence class of $x$ in the quotient $A / A^{-}$, that is:
\begin{equation}
\label{eq:nega-quotient}
    |x| = |y| \iff \exists n \in A^{-}\text{ such that } x \vee n = y \vee n.
\end{equation}

This relation is not a congruence on Kleene algebras, but it is a congruence on distributive lattices. Therefore, $\mathbf{A / A^{-}}$  is a distributive lattice, on which we can define the following binary relation:
\begin{equation}\label{def:contact-quotient}
|x|\ \con_\A\ |y| \;\Longleftrightarrow\;
\forall x' \in |x|,\ \forall y' \in |y|,\ x' \nleq \sim y'.
\end{equation}

In \cite{jalali}, it is shown that the complement of the relation defined in \eqref{def:contact-quotient}, namely:
\begin{equation}
    \label{def:independence}
    |x|\ \notcon_\A\ |y| \;\Longleftrightarrow\;
    \exists x' \in |x|,\ \exists y' \in |y|,\ x' \leq \sim y',
\end{equation}
is an \emph{independence relation}. Consequently, the relation $\con_\A$ is a \emph{contact relation}.

The following result establishes a canonical way to associate a Kleene triple to any Kleene algebra.

\begin{proposition}[\cite{jalali}]
\label{prop:Tfunctor}
Let $\A$ be a Kleene algebra. Then the triple \[
\mathrm{T}(\A)= \langle \mathbf{A / A^{-}} , \con_\A, A^{+}  / A^{-} \rangle\]
is a Kleene triple, where the Boolean filter is given by
\[
A^{+}  / A^{-} = \{\, |x| \in A / A^{-} : x \in A^{+}  \,\}.
\]
Moreover, for every $|x| \in A / A^{-}$, we have
\[
|x|\, \notcon_\A\, |\sim x| \quad \text{and} \quad |x| \vee |\sim x| \in A^{+}  / A^{-}.
\]
\end{proposition}

\begin{proposition}
\label{prop:Tfunctor-mor}
Let $f \colon \A \to \A'$ be a morphism of Kleene algebras. Then the map $\mathrm{T}(f) \colon \mathrm{T}(\A) \to \mathrm{T}(\A')$ defined by
\[
\mathrm{T}(f)(|x|) := |f(x)|
\]
is a well-defined morphism of Kleene triples.
\end{proposition}

Consequently, the assignments
\[
\A\mapsto\mathrm{T}(\A),
\qquad
f\mapsto\mathrm{T}(f),
\]
define a functor $\mathrm{T}:\mathbf{Kl}\to\mathbf{Kt}$. The main result stated by Jalali \cite{jalali} is that the functors $\mathrm{K}$ and $\mathrm{T}$ establish an equivalence of categories. A detailed proof can also be found in \cite{viglizzo99algebras}.

\section{Modal Kleene Algebras}
\label{sec:modal}

In this section, we extend the construction of Kleene algebras from contact lattices by incorporating modal operations. Starting from a modal contact lattice, we define a corresponding twist structure equipped with modal operators. This approach generalizes previous constructions and provides an algebraic framework suitable for the study of modal logics.

\begin{definition}
    A \emph{modal lattice} is an algebra $\D=\langle D, \vee, \wedge,\Box, \Diamond, 0,1\rangle$ such that the reduct $\langle D, \vee, \wedge, 0,1\rangle$ is a bounded distributive lattice, and $\Box, \Diamond,$ are unary operators on $D$, and for all $a, b\in D$, the following equations are satisﬁed:
    \begin{align}
        \label{ml1}
        &\Box(a \wedge b) = \Box a \wedge \Box b;  \tag{ML1}\\
        \label{ml2}
        &\Diamond(a \vee b) = \Diamond a \vee \Diamond b; \tag{ML2}\\
        \label{ml3}
        &\Diamond 0=0. \tag{ML3}
    \end{align}
\end{definition}

When a modal lattice is equipped with a contact relation, we impose an additional compatibility condition. This condition, proposed by Menchón and Rodriguez \cite{Menchon2025} for the overlap relation, is essential for defining modal operators on the twist structure.

\begin{definition}
A \emph{modal contact lattice} is a pair $\langle \D,\con\rangle$, where $\D$ is a modal lattice and $\con$ is a contact relation on $D$ satisfying the following compatibility condition, for all $a,b\in D$:
\begin{equation}
\label{ax:modal-contact}\tag{MCL}
\Box a \con \Diamond b \Longrightarrow a \con b.
\end{equation}
\end{definition}

A typical example of a modal contact lattice is provided by modal Boolean contact algebras, in which the modal operators are interdefinable. In this context, Condition~\eqref{ax:modal-contact} can be reformulated in terms of \emph{subordinations}, which are binary relations. In particular, the standard \emph{non-tangential part} relation is defined by
\[
a \ll b \;\Longleftrightarrow\; a \notcon -b,
\]
where $-b$ denotes the Boolean complement of $b$. In terms of this relation, Condition \eqref{ax:modal-contact} can be restated as
\[
a \ll b \ \Longrightarrow\ \Box a \ll \Box b.
\]
or equivalently,
\[
a \ll b \ \Longrightarrow\ \Diamond a \ll \Diamond b.
\]

This condition has been studied for proximities within the framework of modal De Vries algebras, where it is referred to as $\Diamond$-proximity preserving, a notion introduced in \cite{bezhanishvili2015modal}.

Taking into account the twist construction, we introduce the following definition.

\begin{definition}
\label{def:modal-triple}
A \emph{modal Kleene triple} is a structure $\T=\langle \D, \con, F\rangle$ where $\langle \D, \con \rangle$ is a modal contact lattice, $F$ is a Boolean filter of $\D$ and the following conditions are satisfied:
\begin{align}
    \label{cond:existence}
    &\forall a \in D,\ \exists b \in D \text{ such that } a\notcon b \text{ and } a \vee b \in F;\\
    \label{cond:filter}
    &\text{If } a\notcon b \ \text{ and } a \vee b \in F, \text{ then } \Box a \vee \Diamond b \in F.
\end{align}
\end{definition}

Condition \eqref{cond:existence} is precisely the defining property of a Kleene triple, while \eqref{cond:filter} is a new compatibility requirement that guarantees the modal operators are well-defined on the corresponding twist structure.

The \emph{category of modal Kleene triples}, denoted by $\mathbf{mKt}$, is the category whose objects are modal Kleene triples and whose morphisms $f:\T_1\to\T_2$, where $\T_1=\langle\D_1,\con_1,F_1\rangle$ and $\T_2=\langle\D_2,\con_2,F_2\rangle$, are morphisms of modal bounded distributive lattices that are also morphisms of Kleene triples; that is,
$f[F_1]\subseteq F_2$ and
$f(a) \con_2 f(b)$ implies $a\con_1 b$.

\begin{remark}\label{remark filter}
Let $\T=\langle \D,\con,F\rangle$ be a modal Kleene triple. Then $F$ is an open filter, that is, if $a\in F$, then $\Box a\in F$. Indeed, let $a\in F$. Since $a\notcon 0$, $\Box a=\Box a\vee\Diamond 0\in F$.

Moreover, Dunn's axiom
\[
\Box(a\vee b)\le \Box a\vee\Diamond b
\]
provides a sufficient condition for a Kleene triple to be modal. Indeed, if $\T=\langle \D,\con,F\rangle$ is a Kleene triple equipped with modals operators satisfying Dunn's axiom and $F$ is an open filter, then $\T$ is a modal Kleene triple.
\end{remark}

\begin{definition}
\label{def:modal-kleene}
A \emph{modal Kleene algebra} is an algebra $\A=\langle A,\wedge, \vee, \sim, \Box, 0,1\rangle$ such that its reduct $\langle A,\wedge, \vee, \sim, 0,1\rangle$ is a Kleene algebra and the following conditions hold:
    \begin{align}
        \label{mk}
        &\Box(x \wedge y) = \Box x \wedge \Box y;\tag{MK}\\
        \label{ax:Pos}
        &\sim\Box(z \vee \sim z)\leq \Box( z \vee \sim z);\tag{Pos}\\
        \label{ax:Con}
        &\Box(x \vee (z \wedge \sim z)) \leq \Box x\vee  \sim \Box(x \vee (z \wedge \sim z)).\tag{Con}
    \end{align}
\end{definition}

We denote by $\mathbf{mKl}$ the category whose objects are modal Kleene algebras and whose morphisms are Kleene morphisms that preserve the modal operator.

It is worth noting that the class of modal Kleene algebras forms a variety. The operator $\Box$ is a modal operator that preserves finite meets and, by condition \eqref{ax:Pos}, maps positive elements to positive elements. As usual, we define the dual operator $\Diamond$ by $\Diamond x := \sim \Box \sim x$. Then, the following conditions hold:
    \begin{align}
        \label{ax:MK'}
        &\Diamond(x \vee y) = \Diamond x \vee \Diamond y;\tag{MK'}\\
        \label{ax:Neg}
        &\Diamond(z \wedge \sim z)\leq \sim \Diamond(z \wedge \sim z).\tag{Neg}
    \end{align}

Condition \eqref{ax:Con}, which is equivalent to 
\begin{equation}
    \Box(x \vee (z \wedge \sim z))= \Box x\vee \big(\Box(x \vee (z \wedge \sim z))\wedge\sim \Box(x \vee (z \wedge \sim z))\big),\tag{Con'}\label{ax:Con'}
\end{equation} 
plays a crucial role in ensuring the well-definedness of the operator $\Box$ on the quotient algebra. More precisely, we have the following result.

\begin{theorem}
    Let $\A$ be a modal Kleene algebra. Then the quotient algebra $\mathbf{A/A^{-}}$, equipped with the modal operators given by
    \[
    \Box |x| := |\Box x|, \quad \text{and} \quad \Diamond |x| := |\Diamond x|,
    \]
    is a modal lattice.
\end{theorem}
\begin{proof}
It remains to verify that the relation is compatible with the modal operators. 
Suppose $|x| = |y|$. Then there exists $n \in A^{-}$ such that $x \vee n = y \vee n$. 
Since $\Diamond$ preserves joins, we have $\Diamond x \vee \Diamond n = \Diamond y \vee \Diamond n$. Moreover, by \eqref{ax:Neg} we have $\Diamond n \in A^{-}$. Consequently, $|\Diamond x| = |\Diamond y|$.

To show compatibility for $\Box$, we use \eqref{ax:Con'} together with the fact that $n \wedge \sim n = n$:
    \begin{align*}
        \Box x \vee \bigl[\Box(x\vee n)\wedge \sim \Box(x\vee n)\bigr] &= \Box (x\vee n)\\
        &= \Box (y\vee n)\\
        &= \Box y \vee \bigl[\Box(y \vee n)\wedge \sim \Box(y\vee n)\bigr].
    \end{align*}
    It is straightforward that $\Box(x\vee n)\wedge \sim \Box(x\vee n)=\Box(y\vee n)\wedge \sim \Box(y\vee n)\in A^-$. Hence $|\Box x| = |\Box y|$, establishing compatibility.

It is immediate that axioms \eqref{ml1} and \eqref{ml2} are satisfied. For \eqref{ml3}, note that $\Diamond |0| = |\Diamond 0| = |0|$, since $\Diamond 0 \in A^{-}$.
\end{proof}

The previous result shows that the quotient of a modal Kleene algebra by its negative elements is a modal lattice. We next show that the induced contact relation and filter endow this quotient with the structure of a modal Kleene triple.

\begin{theorem}
\label{thm:cociente-modal}
    If $\A$ is a modal Kleene algebra, then 
    \[
        \mathrm{T}(\A)=\langle \A/\A^{-}, \con_{\A},A^{+}  / A^{-} \rangle
    \]
    is a modal Kleene triple. 
\end{theorem}
\begin{proof}
    By Proposition \ref{prop:Tfunctor}, it suffices to verify Condition~\eqref{cond:filter} and the Modal Contact Condition~\eqref{ax:modal-contact}.

    We first prove the Modal Contact Condition. Suppose that $|x| \notcon_\A |y|$. Then there exist $x' \in |x|$ and $y' \in |y|$ such that $x' \leq \sim y'$. Consequently, $\Box x' \leq \Box \sim y' = \sim \Diamond y'$, and hence $|\Box x| \notcon_\A |\Diamond y|$.

    It remains to check Condition~\eqref{cond:filter}. Assume $|x| \notcon_\A |y|$ and $|x\vee y| \in A^{+}/A^{-}$. Then there exist $x' \in |x|$, $y' \in |y|$, $n \in A^{-}$, and $p \in A^{+}$ such that $x' \leq \sim y'$ and $x' \vee y' \vee n = p$. Since $\Box$ preserves positives, we have $\Box p \in A^{+}$. Moreover, a direct application of \eqref{ax:Con'} yields:
    \[
        \Box p = \Box(x'\vee y'\vee n) = \Box(x'\vee y') \vee [\Box p \wedge \sim \Box p] = \Box(x'\vee y') \vee \sim \Box p,
    \]
    which implies $\sim \Box p = \sim \Box(x'\vee y') \wedge \Box p \in A^{-}$. Distributivity then yields:
    \begin{equation*}
        \Box x' \vee \Diamond y' \vee \sim \Box p = \bigl[\Box x' \vee \Diamond y' \vee \sim \Box(x'\vee y')\bigr] \wedge \bigl[\Box x' \vee \Diamond y' \vee \Box p\bigr].
    \end{equation*}

    Observe that $\Box x' \vee \Diamond y' \vee \Box p \in A^{+}$, since $\Box p \in A^{+}$. Set 
    \[
        u := \Box x' \vee \Diamond y' \vee \sim \Box(x'\vee y').
    \]
    We show that $u \in A^{+}$ by proving $\sim u \leq u$. Indeed, a straightforward computation using \eqref{ax:Con} and $x'\leq \sim y'$ shows that:
    \begin{align*}
        \sim u &\leq \sim \Diamond y' \wedge \Box(x'\vee y') \\
        &= \Box\bigl(\sim y' \wedge (x'\vee y')\bigr) \\
        &= \Box\bigl[x' \vee (\sim y' \wedge y')\bigr] \\
        &\leq \Box x' \vee \sim \Box\bigl[x' \vee (\sim y' \wedge y')\bigr] \\
        &= \Box x' \vee \Diamond y' \vee \sim \Box(x'\vee y') \\
        &= u.
    \end{align*}

    Finally, taking $n_1 := \sim \Box p \in A^{-}$ and $p_1 := u \wedge \bigl[\Box x' \vee \Diamond y' \vee \Box p\bigr] \in A^{+}$, we obtain $\Box x' \vee \Diamond y' \vee n_1 = p_1$. Thus $|\Box x \vee \Diamond y| \in A^{+}/A^{-}$, completing the proof.
\end{proof}

We now show that a homomorphism of modal Kleene algebras induces a morphism between the associated modal Kleene triples.

\begin{proposition}\label{moorfismos mkl}
    If $f : \A \rightarrow \A'$ is a homomorphism of modal Kleene algebras, then the map $\mathrm{T}(f) :\mathrm{T}(\A) \rightarrow \mathrm{T}(\A')$ defined by $\mathrm{T}(f)(|x|) := |f(x)|$ is a morphism in the category of modal Kleene triples.
\end{proposition}
\begin{proof}
    By Proposition \ref{prop:Tfunctor-mor}, $\mathrm{T}(f)$ is a morphism in the category of Kleene triples. It remains to verify that $\mathrm{T}(f)$ preserves the modal operations. Since $f$ preserves $\Box$ and, by duality, also preserves $\Diamond$, we obtain:
    \begin{align*}
        \mathrm{T}(f)(\Box|x|) &= |f(\Box x)| = |\Box f(x)| = \Box \mathrm{T}(f)(|x|),\\
        \mathrm{T}(f)(\Diamond|x|) &= |f(\Diamond x)| = |\Diamond f(x)| = \Diamond \mathrm{T}(f)(|x|).
    \end{align*}
    Thus $\mathrm{T}(f)$ is a morphism of modal Kleene triples.
\end{proof}

The construction of the functor $\mathrm{T}$ extends naturally to the modal setting. More precisely, the assignments
\[
\A\mapsto\mathrm{T}(\A),\qquad
f\mapsto\mathrm{T}(f),
\]
define a functor $\mathrm{T}:\mathbf{mKl}\longrightarrow\mathbf{mKt}$.

Conversely, starting from a modal Kleene triple, we can construct a modal Kleene algebra via the twist construction. More precisely, we have the following theorem.

\begin{theorem}
\label{thm:twist-modal-kleene}
   Let $\T=\langle\D, \con, F\rangle$ be a modal Kleene triple. Then 
   \[\mathbf{K}(\T)=\langle\KC{D}, \cap, \cup, \sim, \boxn, \zero, \one\rangle\] is a modal Kleene algebra, where the modal operator is defined by
   \[
        \boxn(a,b) := (\Box a, \Diamond b).
    \]
    Moreover, the corresponding dual operator is given by $ \diamon(a,b) = (\Diamond a, \Box b)$.
\end{theorem}
\begin{proof}
    By Proposition \ref{prop:twist-kleene}, $\langle\KC{D}, \cap, \cup, \sim, \zero, \one\rangle$ is a Kleene algebra.
   We first verify that the modal operators are well-defined. Take $(a, b) \in \KC{D}$, i.e., $a\notcon b$ and $a\vee b \in F$. Condition \eqref{ax:modal-contact} yields $\Box a\notcon \Diamond b$, and by \eqref{cond:filter} we have $\Box a\vee \Diamond b\in F$; hence $\boxn(a, b) \in \KC{D}$.
   
    It remains to verify that the modal operation satisfies the required properties. Let $(a,b), (e,d) \in \KC{D}$. Since Conditions \eqref{ml1} and \eqref{ml2} hold in $\D$, Equation \eqref{mk} follows directly.
    
    To complete the proof, it remains to verify that Conditions \eqref{ax:Pos} and \eqref{ax:Con'} are satisfied. To this end, observe that every element of $\KC{D}^{+}$ is of the form $(e,0)$. Hence, $\boxn (e,0)= ( \Box e,\Diamond0)=(\Box e,0)\in \KC{D}^{+}.$
 
    Now, let $(a,b), (c,d)\in \KC{D}$ and $(0,e)\in \KC{D}^{-}$. We need to prove that 
     \[
    \boxn ((a,b)\cup (0,e))= \boxn (a,b)\cup [\boxn ((a,b)\cup (0,e))\cap \sim\boxn ((a,b)\cup (0,e))].
    \]
    On the one hand, we have
    \[
    \boxn ((a,b)\cup (0,e))= \boxn(a, b\wedge e)=(\Box a, \Diamond(b \wedge e)),
    \]
    and on the other hand, using the fact the modal operators are isotone in $\D$, we obtain
    \begin{align*}
        \boxn (a,b)\cup \big((\Box a, \Diamond(b \wedge e))\cap \sim(\Box a, \Diamond(b \wedge e))\big)
        &= (\Box a,\Diamond b) \cup (0, \Box a\vee\Diamond(b \wedge e))\\
        &= (\Box a,\Diamond b\wedge(\Box a\vee\Diamond(b \wedge e)))\\
        &= (\Box a,(\Diamond b\wedge\Box a)\vee  \Diamond(b \wedge e))\\
        &= (\Box a, \Diamond(b \wedge e)).
    \end{align*}
    
    Since the two expressions coincide, we conclude that \eqref{ax:Con'} holds, and hence so does \eqref{ax:Con}. 
\end{proof}

We now show that a morphism of modal Kleene triples induces a homomorphism of the corresponding modal Kleene algebras.

\begin{proposition}\label{morfismo mkt}
    Let $\T_1=\langle \D_1, \con_1, F_1 \rangle$ and $\T_2=\langle \D_2, \con_2, F_2 \rangle$ be modal Kleene triples. If $f: \T_1 \rightarrow \T_2$ is a morphism in $\mathbf{mKt}$, then the function $\mathrm{K}(f): \mathbf{K}(\T_1) \rightarrow \mathbf{K}(\T_2)$ given by $\mathrm{K}(f)(a,b) := (f(a), f(b))$ is a morphism in the category $\mathbf{mKl}$ of modal Kleene algebras.
\end{proposition}
\begin{proof}    
    The preservation of Kleene operations follows from Proposition~\ref{prop:Kfunctor-mor}. It remains to check the modal operator:
    \[
        \mathrm{K}(f)(\boxn(a,b)) = (f(\Box a), f(\Diamond b)) = (\Box f(a), \Diamond f(b)) = \boxn(\mathrm{K}(f)(a,b)).
    \]
    Hence, $\mathrm{K}(f)$ is a morphism in $\mathbf{mKl}$.
\end{proof}

The functor $\mathrm{K}$ also extends naturally to the modal setting. Indeed, the assignments
\[ \mathcal{T}\mapsto\mathrm{K}(\mathcal{T}),\qquad f\mapsto\mathrm{K}(f), \] 
define a functor $\mathrm{K}\colon\mathbf{mKt}\to\mathbf{mKl}$.

The following theorems extend the isomorphisms established in \cite{jalali} to the setting of modal Kleene algebras, thereby contributing to the formulation of a categorical equivalence that incorporates the modal structure. First, we show that every modal Kleene algebra is isomorphic to the twist structure arising from its associated modal Kleene triple.

\begin{theorem}
\label{thm:beta-iso}
    Let $\A = \langle A, \wedge, \vee, \sim, \Box, 0, 1 \rangle$ be a modal Kleene algebra, and let $\mathrm{T}(\A)=\langle\A/\A^{-}, \con_{\A},A^{+}  / A^{-} \rangle$ be its associated modal Kleene triple. Then the map
    \[
        \beta_\A \colon \A \to \mathbf{K}(\mathrm{T}(\A)), \quad \beta_\A(x) := (|x|, |\sim x|)
    \]
    is an isomorphism of modal Kleene algebras.
\end{theorem}
\begin{proof}
    Proposition~\ref{prop:Tfunctor} guarantees that $\beta_\A$ is well defined.  Jalali \cite{jalali} proved that $\beta_\A$ is an isomorphism of Kleene algebras; therefore it suffices to show that it preserves the modal operator, which follows directly from the definitions:
    \begin{align*}
        \beta_\A(\Box x) &= (|\Box x|, |\sim \Box x|) = (|\Box x|, |\Diamond (\sim x)|) = \boxn  \beta_\A(x).
    \end{align*}  
\end{proof}

Conversely, starting from a modal Kleene triple, its associated twist algebra, when quotiented by its negative elements, recovers the original triple up to isomorphism.

\begin{theorem}
\label{thm:alpha-iso}
    Let $\T=\langle\D, \con, F\rangle$ be a modal Kleene triple and let $\mathbf{K}(\T)$ be its associated modal Kleene algebra. Then the map $\alpha_\T \colon \mathrm{T}(\mathbf{K}(\T))\to \T$ given by
    \[
        \alpha_\T(|(a,b)|) := a
    \]
    is an isomorphism of modal bounded distributive lattices satisfying
    \[
        |(a,b)|\con_{\mathbf{K}(\T)} |(c,d)| \;\Longleftrightarrow\; a\con c,
    \]
    and $\alpha_\T[\KC{D}^+/\KC{D}^-]=F$.
\end{theorem}
\begin{proof}
    That $\alpha_\T$ is a lattice isomorphism which satisfies the two additional properties is established in \cite{jalali}.  To see that $\alpha_\T$ respects the modal structure, we compute
    \[
        \alpha_\T(|\boxn (a, b)|) = \alpha_\T(| (\Box a, \Diamond b)|) = \Box a = \Box \alpha_\T(|(a,b)|),
    \]
    and similarly $\alpha_\T(|\diamon (a,b)|) = \Diamond a = \Diamond \alpha_\T(|(a,b)|)$. Thus $\alpha_\T$ commutes with the modal operators and is therefore an isomorphism in the category $\mathbf{mKt}$.
\end{proof}

\begin{theorem}\label{teorema equivalencia}
The functors $\mathrm{K}:\mathbf{mKt}\to\mathbf{mKl}$ and $\mathrm{T}:\mathbf{mKl}\to\mathbf{mKt}$
establish a categorical equivalence between the categories $\mathbf{mKt}$ and $\mathbf{mKl}$.
\end{theorem}
\begin{proof}
The result follows immediately from the categorical equivalence in the non-modal setting together with Theorems \ref{thm:cociente-modal}, \ref{thm:twist-modal-kleene}, Propositions \ref{moorfismos mkl}, \ref{morfismo mkt} and Theorems \ref{thm:beta-iso} and \ref{thm:alpha-iso}, which show that the constructions defining the functors preserve the modal structure.
\end{proof}

Although normality is not required for the representation theorems, the modal operators most commonly studied in the literature are normal, that is, they satisfy $\Box 1=1$. This is the case, for instance, in normal modal logic. It is therefore natural to ask how this additional condition is reflected by the representation.

\begin{proposition}
Let $\T=\langle\D,\con,F\rangle$ be a modal Kleene triple such that the modal lattice $\D$ satisfies $\Box 1=1$. Then the associated modal Kleene algebra $\mathrm{K}(\T)$ satisfies $\boxn \one=\one$.
\end{proposition}
\begin{proof}
It follows immediately from the definition of the modal operators on $\mathrm{K}(\T)$. 
\end{proof}

\begin{proposition}
Let $\A$ be a modal Kleene algebra satisfying $\Box 1=1$. Then the underlying modal lattice $\mathbf{A/A^{-}}$ of $\mathrm{T}(\A)$ satisfies $\Box |1|=|1|$.
\end{proposition}
\begin{proof}
By the definition of the modal operator on $\mathbf{A/A^{-}}$, $\Box|1|=|\Box1|=|1|$,
where the last equality follows from the assumption $\Box1=1$.
\end{proof}

Let $\mathbf{NmKl}$ be the full subcategory of $\mathbf{mKl}$ whose objects are normal modal Kleene algebras, and let $\mathbf{NmKt}$ be the full subcategory of $\mathbf{mKt}$ whose objects are modal Kleene triples whose underlying modal lattice satisfies $\Box1=1$. As an immediate corollary of the categorical equivalence, we obtain the following.

\begin{corollary}
The functors $\mathrm{K}:\mathbf{mKt}\to\mathbf{mKl}$ and $\mathrm{T}:\mathbf{mKl}\to\mathbf{mKt}$ restrict to functors $\mathrm{K}:\mathbf{NmKt}\to\mathbf{NmKl}$ and $\mathrm{T}:\mathbf{NmKl}\to\mathbf{NmKt}$, which establish a categorical equivalence between $\mathbf{NmKt}$ and $\mathbf{NmKl}$.
\end{corollary}

 \subsection{Centered modal Kleene algebras}\label{subsec:centered}

The centered case deserves special attention for two reasons. First, every modal Kleene algebra can be embedded into a modal Kleene algebra with a center. Second, its axiomatization becomes significantly simpler. Since centered modal Kleene algebras are considered in the expanded signature obtained by adding the distinguished constant $\cn$, Conditions~\eqref{ax:Pos} and~\eqref{ax:Con} can be replaced by a single identity.

 \begin{lemma}[\cite{cignoli1986class}]
 \label{lem:pos-eq-quotient}
     Let $\A$ be a modal Kleene algebra. Then $\A$ has a center if and only if $A/A^{-} = A^{+} /A^{-}$.
 \end{lemma}

If we consider $F = D$, then any modal contact lattice $\langle \D, \con \rangle$ gives rise to the triple $\mathcal{T}=\langle \D, \con, D \rangle$. Conditions \eqref{cond:existence} and \eqref{cond:filter} are trivially satisfied in this case, so the triple is a modal Kleene triple. Moreover, $\mathbf{K}(\mathcal{T})$ has a center, the element $(0,0)$.

Let $\A$ be a modal Kleene algebra and let $\mathcal{T}=\langle \A/\A^{-}, \con_\A,A/A^{-}\rangle$. Since $\mathbf{K}(\mathrm{T}(\A))$ is a subalgebra of $\mathbf{K}(\mathcal{T})$ and $\mathbf{K}(\mathrm{T}(\A))\cong\A$, it follows that $\A$ embeds into the modal Kleene algebra $\mathbf{K}(\mathcal{T})$. Moreover, by the previous observation, $\mathbf{K}(\mathcal{T})$ has a center.

In this context, the filter component becomes redundant. This allows us to replace modal Kleene triples by modal contact lattices. As a consequence, we obtain the following categorical equivalence between centered modal Kleene algebras and modal contact lattices. The \emph{category of centered modal Kleene algebras}, denoted by $\mathbf{cmKl}$, is the category whose objects are centered modal Kleene
algebras and whose morphisms are homomorphisms of centered modal Kleene algebras. The \emph{category of modal contact lattices}, denoted by $\mathbf{mCL}$, is the category whose objects are modal contact lattices and whose morphisms $f:\langle\D_1,\con_1\rangle\to\langle\D_2,\con_2\rangle$ are morphisms of modal bounded distributive lattices such that
\[
f(a)\,\con_2\,f(b)\Longrightarrow a\,\con_1\,b.
\]
We define the functor $\mathrm{L}:\mathbf{cmKl}\longrightarrow\mathbf{mCL}$
by setting $\mathrm{L}(\A):=\langle \A/\A^{-},\con_\A\rangle$ on objects. On morphisms, if $f:\A\to\A'$
is a homomorphism of centered modal Kleene algebras, then $\mathrm{L}(f)(|x|):=|f(x)|$.

Conversely, we define the functor $\mathrm{K}:\mathbf{mCL}\longrightarrow\mathbf{cmKl}$
by setting for $\mathcal{L}=\langle\D,\con\rangle $, $K(\mathcal{L}):=\langle K_{\con}(D),\cap,\cup,\sim ,\boxn,\cbf,\zero,\one\rangle$ where 
\[
K_{\con}(D)=\{(a,b)\in D\times D: a\notcon b\},
\]
and the operations are those defined in \ref{eq:twist-operations} and Theorem \ref{thm:twist-modal-kleene}. On morphisms, if $f:\langle\D_1,\con_1\rangle\to\langle\D_2,\con_2\rangle$
is a morphism of modal contact lattices, then
\[
\mathrm{K}(f)(a,b):=(f(a),f(b)).
\]

\begin{theorem}
    The functors $\mathrm{K}:\mathbf{mCL}\longrightarrow\mathbf{cmKl}$ and $\mathrm{L}:\mathbf{cmKl}\longrightarrow\mathbf{mCL}$ establish a categorical equivalence between the category $\mathbf{mCL}$ of modal contact lattices and the category $\mathbf{cmKl}$ of centered modal Kleene algebras.
\end{theorem}
\begin{proof}
The functors $\mathrm{K}$ and $\mathrm{L}$ are well defined on the categories
$\mathbf{mCL}$ and $\mathbf{cmKl}$. Indeed, if
$\mathcal{L}=\langle\D,\con\rangle$, then
$\mathrm{K}(\mathcal{L})$ is obtained from
$\mathbf{K}(\mathcal{T})$, where
$\mathcal{T}=\langle\D,\con,D\rangle$, by expanding it with the constant
$\cbf=(0,0)$. Moreover, if $f$ is a morphism of modal contact lattices, then
$\mathrm{K}(f)(0,0)=(0,0)$,
so $\mathrm{K}(f)$ is a homomorphism of centered modal Kleene algebras.

Finally, the natural isomorphisms of
Theorems~\ref{thm:beta-iso} and~\ref{thm:alpha-iso}
remain valid in the expanded signature, since they preserve the distinguished
constant. Therefore, $\mathrm{K}$ and $\mathrm{L}$ establish a categorical
equivalence between $\mathbf{mCL}$ and $\mathbf{cmKl}$.
\end{proof}

Besides simplifying the representation, the presence of a distinguished center also simplifies the equational axiomatization. 

 \begin{proposition}
     Let $\mathbf{A}=\langle A,\wedge,\vee,\sim,\Box,\cn,0,1\rangle$ be a centered modal Kleene algebra that satisfies \eqref{mk}. Then, $\mathbf{A}$ satisfies Conditions \eqref{ax:Pos} and \eqref{ax:Con} if and only if it satisfies the following equation:
    \begin{equation}
        \Box(x\vee \cn)=\Box x\vee \cn.\tag{Cen}\label{centered}
    \end{equation} 
 \end{proposition}
 \begin{proof}
    Assume first that Conditions \eqref{ax:Pos} and \eqref{ax:Con} hold. Then, $\Box(x\vee \cn)\leq \Box x\vee \sim \Box(x\vee \cn)$. Since, $\cn \leq x\vee \cn$, the element $x\vee \cn$ is positive. Hence, by Condition \eqref{ax:Pos}, $\Box(x\vee \cn)$ is also a positive element. Therefore, $\sim \Box(x\vee \cn)\leq \cn$, and consequently $\Box(x\vee \cn)\leq \Box x\vee \cn$. The converse inequality follows, again, from $\cn\leq \Box(x\vee \cn)$ and by Condition \eqref{mk}, $\Box x\leq \Box(x\vee \cn)$.
 
     Conversely, assume that~\eqref{centered} holds. Let $z\in A$. Then, $\cn\leq z\vee\sim z$ and by \eqref{centered}, it follows that $\Box(z\vee\sim z)=\Box(z\vee\sim z)\vee \cn$. Hence, $\cn\leq \Box(z\vee\sim z)$. Thus, Condition~\eqref{ax:Pos} holds.

     Now let $z\in A$. Since $z\wedge\sim z\leq \cn$. By Conditions \eqref{mk} and \eqref{centered}, $\Box(x\vee (z\wedge \sim z))\leq \Box(x\vee \cn)=\Box x\vee \cn$. To prove \eqref{ax:Con}, it remains to show that $\Box x\vee \cn\leq \Box x\vee \sim \Box(x\vee(z\wedge\sim z))$. Indeed,
    \begin{align*}
       \sim\Box x \wedge  \Box(x\vee(z\wedge\sim z)) 
        &\le  \sim \Box x\wedge (\Box x\vee \cn)\\
        &=(\Box x\wedge\sim \Box x)\vee(\cn\wedge\sim \Box x)\\
        &=\cn\wedge\sim \Box x\\
        &\le \cn.
    \end{align*}
    Thus, $\cn=\sim \cn\leq \sim ( \sim\Box x\wedge \Box(x\vee(z\wedge\sim z)))=\Box x \vee \sim \Box(x\vee(z\wedge\sim z))$ and the result follows.
 \end{proof}

A natural additional property is that the modal operator $\Box$ preserves the center.

\begin{definition}
A centered modal Kleene algebra $\A$ is said to have a
\emph{center-preserving modal operator} if
\begin{equation}
\label{KC}\tag{KC}
\Box\cn=\cn.
\end{equation}
We denote by $\mathbf{cmKl}_0$ the full subcategory of $\mathbf{cmKl}$ whose objects are centered modal Kleene algebras satisfying \eqref{KC}.
\end{definition}

\begin{definition}
A \emph{$0$-preserving modal contact lattice} is a modal contact lattice
$\langle\D,\con\rangle$ satisfying
\begin{equation}
\label{D}\tag{D}
\Box0=0.
\end{equation}
We denote by $\mathbf{mCL}_0$ the full subcategory of $\mathbf{mCL}$ whose objects are $0$-preserving modal contact lattices.
\end{definition}

The following proposition establishes the expected correspondence.

\begin{proposition}
\label{pro:mCLc-mKLc}
Let $\mathcal{L}=\langle\D,\con\rangle$ be a $0$-preserving modal contact lattice. Then the twist structure $K(\mathcal{L})$
is a centered modal Kleene algebra satisfying \eqref{KC}.
\end{proposition}
\begin{proof}
Recall that $K(\mathcal{L})$ is a centered modal Kleene algebra with center
$\cbf=(0,0)$.  A straightforward computation shows that 
\[
\boxn\cbf
=\boxn(0,0)
=(\Box0,\Diamond0)
=(0,0)
=\cbf,
\]
where the third equality follows from \eqref{D}.  Thus,
$K(\mathcal{L})$ satisfies Condition~\eqref{KC}.
\end{proof}

\begin{proposition}
Let $\A$ be a centered modal Kleene algebra satisfying \eqref{KC}. Then
$\mathrm{L}(\A)$ is a $0$-preserving modal contact lattice.
\end{proposition}
\begin{proof}
Since $\A$ satisfies \eqref{KC}, we have $\Box \cn=\cn$. It is easy to see that $|\cn|=|0|$. Then, $\Box |0|=\Box |\cn|=|\Box \cn|=|\cn|=|0|$.
\end{proof}

Therefore, the categorical equivalence between $\mathbf{cmKl}$ and
$\mathbf{mCL}$ restricts to the following equivalence.

\begin{corollary}
\label{cor:equiv-mKLc-mCLc}
The full subcategories $\mathbf{cmKl}_0$ of $\mathbf{cmKl}$ and
$\mathbf{mCL}_0$ of $\mathbf{mCL}$ are equivalent.
\end{corollary}

\section{Twist-Product Representations of Some Varieties of Kleene Algebras with Implication}\label{sec:implicative}

Let $\mathbf{H}=\langle H,\wedge,\vee,\to,0,1\rangle$ be a Heyting algebra. Defining
\begin{align*}
    \Box_a b &:= a \to b,\\
    \Diamond_a b &:= a \wedge b,
\end{align*}
yields a lattice equipped with a family of modal operators which, for each $a \in H$, satisfies
conditions \eqref{ml1}, \eqref{ml2} and \eqref{ml3}. Indeed, a Heyting algebra can be equivalently characterized as a bounded distributive lattice in which, for every $a\in H$, the map $\Diamond_a$ admits a right adjoint $\Box_a$. The pair $(\Diamond_a,\Box_a)$ forms a monotone Galois connection for every $a\in H$. 

When applying the twist construction to $\mathbf{H}$, the crucial step is to equip this structure with an implication. A natural choice is to define

\[(a,b)\Rightarrow(c,d):= (a\to c,a\wedge d)=(\Box_ac,\Diamond_ad).\]

With this definition, the resulting twist structure is a Nelson algebra \cite{sendlewski1990nelson}.

This observation suggests that the duality
developed in Section~\ref{sec:modal} can be applied to algebras with implication, provided we interpret
the implication as a family of modal operators indexed by the elements of the algebra.

In this section we show that this is indeed possible.
\begin{definition}
A \emph{Right-Meet-Distributive Implication Lattice} (or RMDI-lattice, for short) is an algebra $\D=\langle D, \wedge, \vee,\to 0, 1\rangle$ of type $(2, 2, 2, 0, 0)$ such that $\langle D, \wedge, \vee, 0, 1\rangle$ is a bounded distributive lattice and the implication operation $\to$ satisfies the following identity:
\begin{align*}
    (a \to b) \wedge (a \to d) &= a \to (b \wedge d). \tag{I1}\label{I1}\\
\end{align*}
\end{definition}

From a category-theoretic and algebraic perspective, this identity states that for every fixed element $a \in D$, the mapping $\Box_a: D \to D$ defined by $\Box_ax := a \to x$ is a lower semilattice homomorphim (or meet-homomorphism) from the underlying semilattice $\langle D, \wedge\rangle$ into itself.

It is crucial to note that RMDI-lattices form a strictly larger variety than that of Distributive Lattice Implication algebras (DLI-algebras) or subresiduated lattices. Because the axiom places no constraints on the first argument of the implication, the variety includes pathological structures where the left-hand behavior is completely decoupled from the lattice operations. For instance, the identity does not force the boundary conditions $0 \to a = 1$ or $a \to 1 = 1$, nor does it impose any interaction with the join operation ($\vee$) from the left. Thus, RMDI-lattices provide the minimal algebraic framework for weak implication systems where conditional context preserves information intersection.
 
These algebras, together with some of their best-known subvarieties, are organized in the diagram presented in Figure~\ref{fig:varieties}. 

\begin{figure}[ht]
\centering
\begin{tikzpicture}
\node[](RMDI){RMDI};
\node[below of= RMDI](DLI){$\mathrm{DLI}$};
\node[below of= DLI] (DLIp) {$\mathrm{DLI}^{+}$};
\node[ below left of =DLIp] (DLIp1) {$\mathrm{DLI}^{+}_{1}$};
\node[ below right of= DLIp] (DLIneg) {$\mathrm{DLI}^{+}_{\neg}$};
\node[ below of= DLIp1] (SRL) {$\mathrm{SRL}$};
\node[below right of= SRL] (Heyting) {$\mathrm{Heyting}$};

\draw (RMDI)--(DLI);
\draw (DLI)--(DLIp);
\draw (DLIp) -- (DLIp1);
\draw (DLIp1) -- (SRL);
\draw (DLIp) -- (DLIneg);
\draw (SRL) -- (Heyting);
\draw (DLIneg) -- (Heyting);
\end{tikzpicture}
\caption{Algebraic subvarieties.}
\label{fig:varieties}
\end{figure}
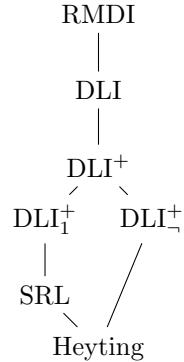

The subvarieties represented in the diagram are characterized by the imposition of the following
additional conditions:

\begin{align}
(a \to d) \wedge (b \to d) &= (a \lor b) \to d,
    \tag{I2}\label{eq:I2}\\
0 \to a &= 1,
    \tag{I3}\label{eq:I3}\\
a \to 1 &= 1,
    \tag{I4}\label{eq:I4}\\
a \wedge (a \to b) &\le b,
    \tag{I5}\label{eq:I5}\\
a \to a &= 1,
    \tag{I6}\label{eq:I6}\\
(a \to b) \wedge (b \to d) &\le a \to d,
    \tag{I7}\label{eq:I7}\\
a \to b &\le d \to (a \to b),
    \tag{I8}\label{eq:I8}\\
b &\le a \to b.
    \tag{I9}\label{eq:I9}
\end{align}
In addition, we consider the following quasi-identity and its equivalent equation under Condition~\eqref{eq:I5}:
\begin{align}
a \wedge b = 0 &\Longrightarrow a \le b \to 0,
    \tag{A}\label{eq:A}\\
a \wedge ((a \wedge b)\to 0) &= a \wedge (b \to 0).
    \tag{A'}\label{eq:Aprime}
\end{align}

Using the notation above, the subvarieties are characterized as follows:

\begin{description}
\item[\textbf{DLI}:] RMDI-lattices satisfying
Conditions~\eqref{eq:I2}--\eqref{eq:I4}.

\item[\textbf{DLI$^{+}$}:] DLI-algebras satisfying
Condition~\eqref{eq:I5}.

\item[\textbf{DLI$^{+}_{1}$}:] DLI$^{+}$-algebras satisfying
Condition~\eqref{eq:I6}.

\item[\textbf{DLI$^{+}_{\neg}$}:] DLI$^{+}$-algebras satisfying
Condition~\eqref{eq:A}. Equivalently, they satisfy
Condition~\eqref{eq:Aprime}.

\item[\textbf{SRL}:] DLI$^{+}_{1}$-algebras satisfying
Conditions~\eqref{eq:I7} and~\eqref{eq:I8}.

\item[\textbf{Heyting}:] SRL-algebras satisfying
Condition~\eqref{eq:I9}.
\end{description}

Now we show how a RMDI-lattice equipped with a contact relation and a Boolean filter gives rise to a family of modal Kleene triples, and hence to a family of modal Kleene algebras via the twist construction.

\begin{definition}
Let $\mathcal{T}=\langle \D, \con, F\rangle$ be a triple where $\D$ is an RMDI-lattice, $\con$ is a contact relation, and $F$ is a Boolean filter. For each $d \in D$, we define the unary modal operators $\Box_d, \Diamond_d: D \to D$ as follows:
\begin{align*}
    \Box_d a &:= d \to a, \\
    \Diamond_d a &:= d \wedge a.
\end{align*}
We say that $\mathcal{T}$ is a \emph{multimodal Kleene triple} if for every $d \in D$, the triple $\mathcal{T}_d=\langle\langle D,\vee,\wedge, \Box_d, \Diamond_d,0,1\rangle, \con, F\rangle$ constitutes a modal Kleene triple; explicitly, this requires that the following conditions hold for all $a, b\in D$:
\begin{align}
\forall a \in D,\ \exists b \in D \text{ such that }& a\notcon b \text{ and } a \vee b \in F;\\
    \big(a\notcon b \text{ and } a \vee b \in F\big) &\implies \Box_d a \vee \Diamond_d b \in F, \label{cond:filter2} \\
    a\notcon b &\implies \Box_d a \notcon \Diamond_d b. \label{ax:modal-contact2}
\end{align}
\end{definition}

We now introduce the corresponding algebraic notion on the side of Kleene algebras. The definition below
is obtained by translating the axioms of a $\mathrm{RMDI}$-lattice into conditions on a family of modal
operators, taking into account the additional requirements needed for the quotient by negative elements to behave properly.

\begin{definition}
A \textbf{Kleene algebra with implication} is an algebra $\A=\langle A,\wedge,\vee,\sim,\Rightarrow,0,1\rangle$ such that $\langle A,\wedge,\vee,\sim,0,1\rangle$ is a Kleene algebra and $\Rightarrow$ is a binary operation satisfying: 
\begin{align}
x\Rightarrow(y\wedge z)
&=(x\Rightarrow y)\wedge(x\Rightarrow z),
\tag{KI1}\label{KI1}\\
\sim(x\Rightarrow(z\vee\sim z))
&\le x\Rightarrow(z\vee\sim z),
\tag{BoxPos}\label{BoxPos}\\
x\Rightarrow\bigl(y\vee(z\wedge\sim z)\bigr)
&\le (x\Rightarrow y)\vee
\sim\!\left(x\Rightarrow\bigl(y\vee(z\wedge\sim z)\bigr)\right),
\tag{BoxCon}\label{BoxCon}\\
x\Rightarrow y
&=(x\vee(z\wedge\sim z))\Rightarrow y,
\tag{wDef}\label{wDef}\\
(x\Rightarrow\sim y)\vee(x\wedge y)
&=\sim(x\Rightarrow\sim y)\vee(x\Rightarrow\sim y),
\tag{D1}\label{D1}\\
(x\Rightarrow\sim y)\wedge(x\wedge y)
&=(x\wedge y)\wedge\sim(x\wedge y).
\tag{D2}\label{D2}
\end{align}
    
\end{definition}

We now reformulate the previous definition in modal terms. Instead of a binary implication, we consider a family of unary modal operators indexed by the underlying algebra. Both presentations are equivalent via the correspondence

\[
\Box_x y:=x\Rightarrow y.
\]

Under this identification, Conditions \eqref{KI1}--\eqref{D2} become
\[
\begin{aligned}
\eqref{KI1}\qquad
&\Box_x(y \wedge w)
= \Box_x y \wedge \Box_x w,\\[1mm]
\eqref{BoxPos}\qquad
&\sim \Box_x (z \vee \sim z)
\leq \Box_x (z \vee \sim z),\\[1mm]
\eqref{BoxCon}\qquad
&\Box_x \bigl(y \vee (z \wedge \sim z)\bigr)
\leq \Box_x y \vee
\sim \Box_x \bigl(y \vee (z \wedge \sim z)\bigr),\\[1mm]
\eqref{wDef}\qquad
&\Box_x y
=\Box_{x\vee(z\wedge\sim z)}y,\\[1mm]
\eqref{D1}\qquad
&\Box_x(\sim y)\vee (x\wedge y)
=\sim \Box_x(\sim y)\vee \Box_x(\sim y),\\[1mm]
\eqref{D2}\qquad
&\Box_x(\sim y)\wedge (x\wedge y)
=(x\wedge y)\wedge \sim (x\wedge y).
\end{aligned}
\]

Conditions \eqref{KI1}, \eqref{BoxPos} and \eqref{BoxCon} are exactly \eqref{mk}, \eqref{ax:Pos} and \eqref{ax:Con} from
Definition~\ref{def:modal-kleene}, adapted to a family of operators. Conditions \eqref{wDef}, $\eqref{D1}$ and $\eqref{D2}$ will
ensure that the quotient $\A/\A^{-}$ inherits a well-defined operations. As usual, we define the dual
operators by $\Diamond_x y := \sim \Box_x \sim y$.

\begin{theorem}
\label{thm:quotient-DLI3}
Let $\A$ be a Kleene algebra with an implication. Then the structure
\[
    \big\langle \langle A/A^{-},\wedge,\vee,\to,|0|,|1|\rangle,\; \con_{\mathbf{A}},\; A^{+}/A^{-} \big\rangle,
\]
where the implication is defined by $|x| \to |y| := |x\Rightarrow y|$, is a multimodal Kleene triple.
\end{theorem}
\begin{proof}
First, we show that the implication is well defined. Fix $x\in A$ and consider the modal operator $\Box_x$. Then, $|x| \to |y|=|\Box_x y|$. By Condition~\eqref{BoxCon}, the induced map on $A/A^{-}$ is well defined with respect to the second argument. It remains to verify that the definition is independent of the choice of the first representative. Suppose that $|x|=|x'|$. Then there exists $n\in A^{-}$ such that $x \vee n = x' \vee n$. Since $n = n \wedge \sim n$, by \eqref{wDef} we have $\Box_x y = \Box_{x \vee n}y=\Box_{x' \vee n}y=\Box_{x'}y$. Thus $|\Box_x y| = |\Box_{x'} y|$, so the definition is independent of the representatives. Moreover, it is straightforward to verify that, by Condition~\eqref{KI1}, $\langle A/A^{-},\wedge,\vee,\to,|0|,|1|\rangle$ is a RMDI-lattice.

Define $\Box_{|x|}|y|:=|\Box_xy|$ and $\Diamond_{|x|}|y|:=|\Diamond_xy|$. It follows immediately from Conditions~\eqref{KI1}--\eqref{BoxCon} that, for each $x\in A$, the structure $\langle A,\wedge,\vee,\sim,\Box_x,0,1\rangle$ is a modal Kleene algebra. By Theorem~\ref{thm:cociente-modal}, we obtain that for each $x\in A$, $$\langle \langle A/A^{-},\wedge,\vee,\Box_{|x|},\Diamond_{|x|},|0|,|1|\rangle, \con_{\mathbf{A}},A^+/A^{-}\rangle$$ is a modal Kleene triple.

Now, we will check that $|\Diamond_xy|=|x|\wedge|y|$. Let us consider $n=((x\wedge y)\wedge \sim (x\wedge y))\vee ((\Diamond_xy\wedge\sim \Diamond_xy))$. $n$ is a negative element because it is the join of two negative elements. Moreover,
\[
\begin{aligned}
\Diamond_xy\vee n
&=\Diamond_xy
  \vee((x\wedge y)\wedge\sim(x\wedge y))\\
&=\Diamond_xy
  \vee(\Box_x(\sim y)\wedge(x\wedge y))
  &&\text{by \eqref{D2}}\\
&=(\Diamond_xy\vee\sim\Diamond_xy)
  \wedge(\Diamond_xy\vee(x\wedge y)).
\end{aligned}
\]

On the other hand,
\[
\begin{aligned}
(x\wedge y)\vee n
&=(x\wedge y)\vee(\Diamond_xy\wedge\sim\Diamond_xy)\\
&=((x\wedge y)\vee\Diamond_xy)
  \wedge((x\wedge y)\vee\sim\Diamond_xy)\\
&=((x\wedge y)\vee\Diamond_xy)
  \wedge(\Diamond_xy\vee\sim\Diamond_xy) &&\text{by \eqref{D1}.}\\
\end{aligned}
\]

It follows that $ \Diamond_xy\vee n=(x\wedge y)\vee n$ and therefore $|\Diamond_x y|=|x\wedge y|=|x|\wedge|y|$.

\end{proof}

Conversely, starting from a multimodal Kleene triple, the twist construction yields a Kleene algebra with an implication.

\begin{theorem}
\label{thm:twist-modal-family}
Let $\mathcal{T}=\langle \D,\con,F\rangle$ be a multimodal Kleene triple. Then \[
\mathbf{K}(\mathcal{T})=\langle\KC{D},\cap,\cup,\sim,\Rightarrow,\zero,\one\rangle
\]
where $(d,e)\Rightarrow (a,b) := (\Box_d a, \Diamond_d b)=(d\rightarrow a,d\wedge b)$, is a Kleene algebra with implication.
\end{theorem}
\begin{proof}
Since for each $d\in D$ the triple $\mathcal{T}_d$ is a modal Kleene triple, $\mathbf{K}(\mathcal{T})$ satisfies \eqref{KI1}, \eqref{BoxPos} and \eqref{BoxCon}. 

It remains to
verify \eqref{wDef}, \eqref{D1} and \eqref{D2} for the family.

For \eqref{wDef}, take $(a,b), (d,e) \in \KC{D}$ and $(0,n)\in \KC{D}^-$. Then $(d,e)\Rightarrow (a,b) = (\Box_d a,\Diamond_d b)=(d,e\wedge n)\Rightarrow (a,b)=((d,e)\cup (0,n))\Rightarrow (a,b)$ and the result follows.

For \eqref{D1}, let $(a,b),(d,e)\in \KC{D}$. Then,
\begin{align*}
((d,e)\Rightarrow\sim (a,b))\cup((d,e)\cap (a,b))
    &=(\Box_db,\Diamond_d a)\cup (d\wedge a,e\vee b)\\
    &=(\Box_db \vee (d\wedge a),\Diamond_d a\wedge(e\vee b))\\
    &=(\Box_db \vee (d\wedge a),0).    
\end{align*}
On the other hand, 
\begin{align*}
    \sim ((d,e)\Rightarrow \sim (a,b))\cup  ((d,e)\Rightarrow \sim (a,b))
    &=\sim (\Box_db,\Diamond_d a)\cup (\Box_db,\Diamond_d a)\\
    &=(\Diamond_d a,\Box_db)\cup (\Box_db,\Diamond_d a)\\
    &=(\Box_db \vee \Diamond_d a,0)\\
    &=(\Box_db \vee (d\wedge a),0). 
\end{align*}

Therefore, \eqref{D1} holds. The proof of \eqref{D2} is analogous.
\end{proof}

The previous constructions naturally give rise to a categorical equivalence. To state it, we first introduce the relevant categories.

The category of multimodal Kleene triples, denoted by $\mathbf{KIt}$, has multimodal Kleene triples as objects. Its morphisms are homomorphisms of RMDI-lattices that are also morphisms of Kleene triples.

The category of Kleene algebras with implication, denoted by $\mathbf{KI}$, has Kleene algebras with implication as objects and homomorphisms preserving all the operations as morphisms.

\begin{theorem}
The categories $\mathbf{KI}$ and $\mathbf{KIt}$ are categorically equivalent.
\end{theorem}
\begin{proof}
    The construction of the functors and the verification of the natural isomorphisms extend those of Theorem~\ref{teorema equivalencia} in a straightforward way by considering the family of modal operators instead of a single modal operator.
\end{proof}

The categorical equivalence established above extends naturally to subvarieties determined by additional equations. More precisely, every subvariety of RDMIs axiomatized by additional equations corresponds, under the translation given in Table~\ref{tab:rdmi-kleene}, to the subvariety of Kleene algebras with implication axiomatized by the corresponding equations.

\begin{table}[ht]
\centering
\small
\renewcommand{\arraystretch}{1.5}
\begin{tabular}{|c|l|}
\hline
RMDI& Kleene algebra with implication \\
\hline
    \eqref{eq:I2} &
$(x\Rightarrow z)\wedge(y\Rightarrow z)=(x\vee y)\Rightarrow z$
\\
\hline
\eqref{eq:I3} &
$0\Rightarrow x=1$
\\
\hline
\eqref{eq:I4} &
$x\Rightarrow1=1$
\\
\hline
\eqref{eq:I5} &
$x\wedge(x\Rightarrow y)\le x\wedge(\sim x\vee y)$
\\
\hline
\eqref{eq:I6} &
$x\Rightarrow x=1$
\\
\hline
\eqref{eq:I7} &
$(x\Rightarrow y)\wedge(y\Rightarrow z)
\le
(x\Rightarrow z)\vee(y\wedge\sim y)$
\\
\hline
\eqref{eq:I8} &
$(x\Rightarrow y)\le z\Rightarrow(x\Rightarrow y)$
\\
\hline
\eqref{eq:I9} &
$ y\le x\Rightarrow y$
\\
\hline
\eqref{eq:Aprime} &
$x\wedge((x\wedge y)\Rightarrow 0)= x\wedge(( y\Rightarrow 0)\vee \sim x)$
\\
\hline
\end{tabular}
\caption{Translations of the additional axioms}
\label{tab:rdmi-kleene}
\end{table}

To facilitate the comparison with previous work, we first prove that, in certain varieties of Kleene algebras with implication, the congruence relation is definable in terms of the implication. The proof follows the general strategy of the proof of the corresponding result in \cite[p.~154]{viglizzo99algebras}; we provide the details here for the sake of completeness.

\begin{proposition}\label{congruenciaImplicacion}
Let $\A=\langle A,\wedge,\vee,\sim,\Rightarrow,0,1\rangle$ be a Kleene algebra with implication satisfying Equations \eqref{KI1}--\eqref{D2} and
\begin{align} 
(x\Rightarrow z)\wedge(y\Rightarrow z) &=(x\vee y)\Rightarrow z,\label{eq:meet}\tag{KI2}\\ 
x\wedge(x\Rightarrow y) &\le x\wedge(\sim x\vee y),\label{eq:compat}\tag{KI5}\\ 
x\Rightarrow x &=1.\label{eq:refl} \tag{KI6}
\end{align}
Then, for every $x,y\in A$,
\[
|x|=|y|
\quad\Longleftrightarrow\quad
x\Rightarrow y=1
\ \text{and}\
y\Rightarrow x=1.
\]
\end{proposition}
\begin{proof}
Suppose first that $|x|=|y|$. Then there exists $n\in A^{-}$ such that $x\vee n=y\vee n$. By Equation~\eqref{BoxCon},
\[
x\Rightarrow(y\vee n)
\leq
(x\Rightarrow y)\vee
\sim\bigl(x\Rightarrow(y\vee n)\bigr).
\]
Moreover,
\begin{align*}
1
&=x\Rightarrow x
=x\Rightarrow\bigl(x\wedge(x\vee n)\bigr)\\
&=(x\Rightarrow x)\wedge(x\Rightarrow(x\vee n))\\
&=x\Rightarrow(x\vee n)
=x\Rightarrow(y\vee n).
\end{align*}
Hence, $1\leq (x\Rightarrow y)\vee0$,
and therefore $x\Rightarrow y=1$. By symmetry, $y\Rightarrow x=1$.

Conversely, suppose that $x\Rightarrow y=1$ and $y\Rightarrow x=1$. By Equation~\eqref{eq:meet},
\[
(x\vee y)\Rightarrow x
=(x\Rightarrow x)\wedge(y\Rightarrow x)
=1.
\]
Using Equation~\eqref{eq:compat},
\[
(x\vee y)\wedge((x\vee y)\Rightarrow x)
\leq
((x\vee y)\wedge\sim(x\vee y))\vee x,
\]
and therefore
\[
x\vee y
\leq
((x\vee y)\wedge\sim(x\vee y))\vee x
\leq
x\vee y.
\]
Hence, $x\vee y=((x\vee y)\wedge\sim(x\vee y))\vee x$. Setting $n=(x\vee y)\wedge\sim(x\vee y)$, we obtain $x\vee n=x\vee y$. By symmetry, $y\vee n=x\vee y$, and consequently $x\vee n=y\vee n$. Therefore, $|x|=|y|$.
\end{proof}

The variety of Kleene algebras with implication satisfying Equations~\eqref{KI1}--\eqref{D2}, together with those corresponding to (I2)--(I8) in Table~\ref{tab:rdmi-kleene}, coincides with the variety of subresiduated Nelson algebras studied in~\cite{noemi2025}. Hence, the categorical equivalence established above yields, as an immediate consequence, a categorical equivalence between the category of modal Kleene triples $\langle\D,\con,F\rangle$, where $\D$ is a subresiduated residuated lattice (SRL), and the category of subresiduated Nelson algebras.

Finally, Propositions~\ref{congruenciaImplicacion} and \ref{prop:DLIneg-forces-overlap} show that in the subvariety $\mathrm{DLI}_{\neg}^+$—and in particular for Heyting algebras—the contact relation of a multimodal Kleene triple is forced to be the overlap. Consequently, for such algebras the duality developed here reduces to the well-known representation of Nelson algebras via Heyting algebras together with the overlap relation. 

\begin{proposition}
\label{prop:DLIneg-forces-overlap}
Let $\D$ be a $\mathrm{DLI}_{\neg}^+$-algebra, and let $\langle\D,\con,F\rangle$ be a multimodal Kleene triple. Then $\con =\Overl$, i.e., $\con$ is the overlap contact relation.
\end{proposition}
\begin{proof}
Assume that $a\con b$. By \eqref{cond:existence}, there exists $a' \in D$ such that
$a\notcon a'$ and $a \vee a' \in F$. In particular, $a \wedge a' = 0$ and by \eqref{ax:modal-contact2}, $\Diamond_a a \notcon \Box_a a'$. Since $D$ is a $\mathrm{DLI}_{\neg}^+$-algebra, we have
\begin{align*}
    \neg a := a \to 0 = a \to (a \wedge a') \overset{\eqref{I1}}{=} (a \to a) \wedge (a \to a') \leq \Box_a a'.
\end{align*}
Now suppose that $a \wedge b = 0$. By condition \eqref{eq:A}, $b \leq \neg a \leq \Box_a a'$. Since $a\con b$, it follows that $a\con \Box_a a'$, and consequently, $\Diamond_a a \con \Box_a a'$, a contradiction. Hence $a\wedge b\neq 0$, and therefore $\con = \Overl$.
\end{proof}

\section{Topological Duality}
\label{sec:topological-duality}

In this section, we extend the topological duality for Kleene triples developed by Jalali~\cite{jalali} to modal Kleene algebras. The section is organized as follows: first, we recall Jalali's duality for Kleene triples; in Section~\ref{subsec:topological-duality-modal-lattices}, we extend Priestley duality to modal lattices with $\Diamond$ and $\Box$ operators; finally, in Section~\ref{subsec:topological-duality-modal-kleene}, we bring these components together to establish the dual equivalence for modal Kleene triples and, consequently, for modal Kleene algebras.

We begin by recalling that a \emph{Priestley space} is a compact ordered topological space $\X = \langle X, \tau, \leq \rangle$ such that for any $x,y \in X$ with $x \not\leq y$, there exists a clopen up-set $V \subseteq X$ containing $x$ but not $y$. Classical Priestley duality establishes a dual equivalence between the category of bounded distributive lattices with homomorphisms and the category of Priestley spaces with continuous order-preserving maps (Priestley morphisms). The remainder of this preliminary discussion, up to Theorem~\ref{thm:Kt-duality}, reproduces the topological duality for Kleene triples established by Jalali~\cite{jalali}; we state the relevant definitions, propositions, and theorems without proof, solely for later reference, and refer the interested reader to~\cite{jalali, viglizzo99algebras} for full details.

\begin{definition}
\label{def:dual-kleene-triple}
A \emph{dual Kleene triple} is a structure $\mathcal{X}=\langle \X, H, \Q \rangle$, where $\X$ is a Priestley space, $H\subseteq X$ is a closed up-set consisting of maximal elements, and $\Q$ is a binary relation on $X$, satisfying:
\begin{itemize}
    \item $\Q$ is reflexive;
    \item $\Q$ is symmetric;
    \item $\Q$ is a closed down-set of $X^2$ with respect to the componentwise order;
    \item $\Q$ is compatible with $H$: if $x\in H$ and $x\Q y$, then $y\le x$.
\end{itemize}

Given dual Kleene triples $\mathcal{X}_1=\langle \X_1,H_1,\Q_1\rangle$ and $\mathcal{X}_2=\langle \X_2,H_2,\Q_2\rangle$, a function $f \colon \X_1 \to \X_2$ is a \emph{dual Kleene morphism} if it is a Priestley morphism satisfying:
\begin{itemize}
    \item if $x \Q_1 y$, then $f(x) \Q_2 f(y)$;
    \item $f[H_1]\subseteq H_2$.
\end{itemize}

The category of dual Kleene triples and dual Kleene morphisms is denoted by $\mathbf{Kt}^{*}$.
\end{definition}

Given a dual Kleene triple $\mathcal{X}=\langle \X, H, \Q \rangle$, we define a binary relation $\con_{\Q}$ on the set of increasing clopen subsets of $\X$, denoted $\U(\X)$, by
\[
    V \con_{\Q} W \iff \exists (x,y) \in V \times W \ \text{such that} \ x\Q y.
\]
It can be shown that $\con_{\Q}$ is a contact relation, which leads to the following result.

\begin{proposition}
\label{prop:dual-to-triple}
Let $\mathcal{X} = \langle \X, H, \Q \rangle$ be a dual Kleene triple. Then $\mathrm{U}(\mathcal{X}) = \langle \U(\X), \con_{\Q}, \mathcal{F}(H) \rangle$ is a Kleene triple, where $\mathcal{F}(H)$ is the filter of all clopen up-sets containing $H$.
\end{proposition}

\begin{proposition}
\label{prop:Uf-Kt-morphism}
If $f \colon \mathcal{X}_1 \to \mathcal{X}_2$ is a $\mathbf{Kt}^{*}$-morphism, then the map
\[
\mathrm{U}(f) \colon \mathrm{U}(\mathcal{X}_2) \to \mathrm{U}(\mathcal{X}_1), \qquad V \mapsto f^{-1}[V]
\]
is a $\mathbf{Kt}$-morphism.
\end{proposition}

Consequently, the assignments
\[
\mathcal{X} \longmapsto \mathrm{U}(\mathcal{X}), \qquad f \longmapsto \mathrm{U}(f),
\]
define a contravariant functor $\mathrm{U} \colon \mathbf{Kt}^{*} \longrightarrow \mathbf{Kt}$.

We now turn to the inverse construction, associating a dual space with each algebra. Every bounded distributive lattice $\D$ gives rise to a Priestley space $\X(\D) = \langle \mathrm{X}(\D), \tau_{\D}, \subseteq \rangle$, where $\mathrm{X}(\D)$ is the set of prime filters of $\D$, ordered by inclusion, and $\tau_{\D}$ is the topology generated by the subbasis
\[
\{\sigma(a) : a \in D\} \cup \{\mathrm{X}(\D) \setminus \sigma(a) : a \in D\},
\]
with $\sigma(a) := \{P \in \mathrm{X}(\D) : a \in P\}$.

\begin{proposition}
\label{prop:triple-to-dual}
Let $\T$ be a Kleene triple. Then $\mathrm{X}(\T) = \langle \X(\D), \Q_{\con}, \mathcal{H}(F) \rangle$ is a dual Kleene triple, where
\begin{equation}
\label{eq:dual-contact}
     P\Q_{\con} P' \iff \forall a \in P \; \forall b \in P' \; (a \con b),\tag{DC}
\end{equation}
and $\mathcal{H}(F) = \{P \in X(\D) : F \subseteq P\}$.
\end{proposition}

\begin{proposition}
\label{prop:Xh-Kt-morphism}
If $h \colon \T_1 \to \T_2$ is a morphism of $\mathbf{Kt}$, then the map
\[
\mathrm{X}(h) \colon \mathrm{X}(\T_2) \to \mathrm{X}(\T_1), \qquad P \mapsto h^{-1}[P]
\]
is a morphism of $\mathbf{Kt}^{*}$.
\end{proposition}

Consequently, the assignments
\[
\T \longmapsto \mathrm{X}(\T), \qquad h \longmapsto \mathrm{X}(h),
\]
define a contravariant functor $\mathrm{X} \colon \mathbf{Kt} \longrightarrow \mathbf{Kt}^{*}$.

The natural isomorphisms establishing the duality are given by the standard evaluation.

\begin{theorem}
\label{thm:epsilon-iso}
Let $\mathcal{X} = \langle \X, H, \Q \rangle$ be a dual Kleene triple. The map
\[
\epsilon \colon X \longrightarrow \mathrm{X}(\mathrm{U}(\mathcal{X})), \qquad x \mapsto \{V \in \U(\X) : x \in V\}
\]
is an isomorphism of dual Kleene triples. That is, $\epsilon$ is a homeomorphism, an order isomorphism, $\epsilon[H] = \mathcal{H}(\mathcal{F}(H))$, and for all $x,y \in X$, $ x\Q y$ if and only if $\epsilon(x) \Q_{\con_{\Q}}\epsilon(y)$.
\end{theorem}

\begin{theorem}
\label{thm:sigma-iso}
Let $\T = \langle \D, \con, F \rangle$ be a Kleene triple. The map
\[\sigma \colon D \longrightarrow \mathrm{U}(\mathrm{X}(\T)), \qquad a \mapsto \{P \in X(\D) : a \in P\},\]
is an isomorphism of Kleene triples. That is, $\sigma$ is a lattice isomorphism, $\sigma[F] = \mathcal{F}(\mathcal{H}(F))$, and for all $a,b \in D$, $a \con b$ if and only if $\sigma(a) \con_{\Q_{\con}} \sigma(b)$.
\end{theorem}

As a direct consequence, we obtain the main duality result of Jalali~\cite{jalali}.

\begin{theorem}
\label{thm:Kt-duality}
The functors $\mathrm{U} \colon \mathbf{Kt}^{*} \to \mathbf{Kt}$ and $\mathrm{X} \colon \mathbf{Kt} \to \mathbf{Kt}^{*}$ establish a dual equivalence between the categories $\mathbf{Kt}$ and $\mathbf{Kt}^{*}$.
\end{theorem}

\subsection{Topological Duality for Modal Lattices}
\label{subsec:topological-duality-modal-lattices}

To extend Priestley duality to modal Kleene algebras, we represent modal operators via binary relations on Priestley spaces. The operator $\Diamond$ is normal, while $\Box$ is non-normal; such operators characterize the class of regular modal logics~\cite[p.~234]{Chellas1980}. Historically, these systems were introduced by Lemmon~\cite{Lemmon1957} to capture notions such as moral obligation or scientific necessity, where the rule of necessitation is omitted, thereby avoiding $\Box \top$ as a logical truth~\cite{Lemmon1957, palmigiano2016}. Kripke~\cite{Kripke1965} provided a relational semantics for non-normal systems using \emph{non-normal worlds} (or \emph{impossible worlds}).

In the topological setting, we adopt a unified approach. For the normal operator $\Diamond$, the existing duality for additive operators~\cite{cignoli1991remarks, petrovich1996} yields a binary relation on the Priestley space. For the non-normal operator $\Box$, we use a binary relation together with a distinguished subset $N$ of normal worlds, mirroring Kripke's impossible-worlds semantics. Combining these devices yields a full topological duality for modal lattices. Since the case of $\Diamond$ has already been fully treated in~\cite{cignoli1991remarks, petrovich1996}, we only recall the relevant definitions below and refer the reader to those papers for proofs; we then develop the dual theory of $\Box$-spaces in detail, as this is the main new component required for our purposes.

\begin{definition}
\label{def:diamond-space}
A \emph{$\Diamond$-space} is a structure $\langle \X, \G \rangle$, where $\X$ is a Priestley space and $\G$ is a binary relation on $X$, satisfying:
\begin{itemize}
    \item for every clopen up-set $V$, the set $\G^{-1}[V]=\{x\in X:\G(x)\cap V\neq\emptyset\}$ is a clopen up-set;
    \item if $x\notG y$, then there exists a clopen up-set $V$ such that $y\in V$ and $\G(x)\cap V=\emptyset$.
\end{itemize}

Given $\Diamond$-spaces $\langle \X_1, \G_1 \rangle$ and $\langle \X_2, \G_2 \rangle$, a function
$f \colon X_1 \to X_2$ is a \emph{$\Diamond$-morphism} if it is a Priestley morphism satisfying:
\begin{itemize}
    \item if $x \G_1 y$, then $f(x) \G_2 f(y)$;
    \item if $f(x) \G_2 z$, then there exists $y \in X_1$ such that $x \G_1 y$ and $z \leq f(y)$.
\end{itemize}
\end{definition}

Given such a space, the modal possibility operator on the lattice of clopen up-sets $\U(\X)$ is defined by $\Diamond_{\G}(V) := \G^{-1}[V]$.

Conversely, for a bounded distributive lattice $\D$ with a normal operator $\Diamond$, its Priestley space of prime filters $\X(\D)$ is equipped with the relation
\[
    P\G_{\Diamond} Q \iff Q \subseteq \Diamond^{-1}[P].
\]

Then the pair $\langle \X(\D),\G_\Diamond\rangle$ is a $\Diamond$-space and there is a dual equivalence between the category of bounded distributive lattices equipped with a normal possibility operator $\Diamond$ and homomorphisms, and the category of $\Diamond$-spaces and $\Diamond$-morphisms.

We now turn to the necessity operator.

\begin{definition}
\label{def:box-space}
A \emph{$\Box$-space} is a structure $\langle \X, \R, N \rangle$, where $\X$ is a Priestley space, $\R$ is a binary relation on $X$, and $N \subseteq X$, satisfying:
\begin{enumerate}
    \item \label{ax:N-clopen} $N$ is a clopen up-set;
    \item \label{ax:boxR-well-def} if $V$ is a clopen up-set, then $\{x \in N : \R(x) \subseteq V\}$ is a clopen up-set;
    \item \label{ax:boxR-epsilon} if $x\notR y$, then $x \in N$ and there exists a clopen up-set $V$ such that $y \notin V$ and $\R(x) \subseteq V$.
\end{enumerate}

Given $\Box$-spaces $\langle \X_1, \R_1, N_1 \rangle$ and $\langle \X_2, \R_2, N_2 \rangle$, a function
$f \colon X_1 \to X_2$ is a \emph{$\Box$-morphism} if it is a Priestley morphism satisfying:
\begin{align}
    &f^{-1}[N_2] = N_1; \tag{BN}\label{ax:box-morph-N}\\
    &\text{If }x\R_1 y \text{ then } f(x)\R_2 f(y); \tag{BR1}\label{ax:box-morph1}\\
    &\text{If }x \in N_1 \text{ and } f(x)\R_2 z \text{, then } \exists y \in X_1 \text{ such that }x\R_1 y\text{ and } f(y) \leq z. \tag{BR2}\label{ax:box-morph2}
\end{align}
\end{definition}

The following proposition establishes that the operator induced by a $\Box$-space preserves finite meets.

\begin{proposition}
\label{prop:box-R-meet-preserving}
Let $\langle \X, \R, N \rangle$ be a $\Box$-space. In the lattice of clopen up-sets $\U(\X)$, define
\[
\Box_{\R}(V) := \{x \in N : \R(x) \subseteq V\}.
\]
Then $\Box_{\R}$ is a unary operator preserving finite meets.
\end{proposition}
\begin{proof}
We first observe that $\Box_{\R}(V)$ is a clopen up-set by condition~\ref{ax:boxR-well-def} of Definition~\ref{def:box-space}. Let $V_1, V_2$ be clopen up-sets. Then
\begin{align*}
\Box_{\R}(V_1 \cap V_2)
&= \{x \in N : \R(x) \subseteq V_1 \cap V_2\} \\
&=\{x \in N : \R(x) \subseteq V_1\} \cap \{x \in N : \R(x) \subseteq V_2\} \\
&= \Box_{\R}(V_1) \cap \Box_{\R}(V_2).
\end{align*}
\end{proof}

Moreover, from Proposition~\ref{prop:box-R-meet-preserving} we obtain
\[
\Box_{\R}(X) = \{x \in N : \R(x) \subseteq X\} = N.
\]

We now show that every bounded distributive lattice with a meet-preserving modal operator gives rise to a $\Box$-space.

\begin{proposition}
\label{prop:XD-box-space}
Let $\D$ be a bounded distributive lattice with a modal operator $\Box$ preserving finite meets, and consider its Priestley space $\X(\D)$ of prime filters. Define the relation $\R_\Box$ by
\[
P\R_\Box Q \in  \iff \Box^{-1}[P] \subseteq Q,
\]
and let $N = \sigma(\Box 1)$. Then:
\begin{enumerate}
    \item For every $a \in D$, $\Box_{\R_\Box}(\sigma(a)) = \sigma(\Box a)$;
    \item If $P\notR_\Box Q$, then $P \in N$ and there exists a clopen up-set $V$ such that $Q \notin V$ and $\R_\Box(P) \subseteq V$.
\end{enumerate}
Therefore $\langle \X(\D), \R_\Box, N \rangle$ is a $\Box$-space.
\end{proposition}
\begin{proof}
\begin{enumerate}
    \item Let $P \in \Box_{\R_\Box}(\sigma(a))$. Then $P \in N$ and $\R_\Box(P) \subseteq \sigma(a)$. We show $\Box a \in P$. Suppose not. Then $a \notin \Box^{-1}[P]$. Since $P \in N$ gives $\Box 1 \in P$, and $\Box$ preserves finite meets, the set $\Box^{-1}[P]$ is a filter. Hence there exists a prime filter $T$ extending it with $a \notin T$. Thus $P \R_\Box T$ and $T \notin \sigma(a)$, contradicting $\R_\Box(P) \subseteq \sigma(a)$. Therefore $\Box a \in P$, so $\Box_{\R_\Box}(\sigma(a)) \subseteq \sigma(\Box a)$.

    The converse is immediate: if $\Box a \in P$, then $a \in \Box^{-1}[P]$; for any $T \in \R_\Box(P)$, we have $a \in T$, so $\R_\Box(P) \subseteq \sigma(a)$; and since $\Box a \leq \Box 1$, we get $P \in N$. Hence $P \in \Box_{\R_\Box}(\sigma(a))$.

    \item If $P\notR_\Box Q$, then $\Box^{-1}[P] \not\subseteq Q$, so choose $a \in \Box^{-1}[P]$ with $a \notin Q$. Then $Q \notin \sigma(a)$ and, as $\Box a \in P$, we have $P \in N$. Moreover, $\R_\Box(P) \subseteq \sigma(a)$; indeed, if $T \in \R_\Box(P)$, then $\Box^{-1}[P] \subseteq T$, hence $a \in T$. Taking $V = \sigma(a)$ yields the desired result.
\end{enumerate}
Thus $\langle \X(\D), \R_\Box, N \rangle$ satisfies the conditions of Definition~\ref{def:box-space}.
\end{proof}

\begin{proposition}
\label{prop:Uf-box-morphism}
Let $\langle \X_1, \R_1, N_1 \rangle$ and $\langle \X_2, \R_2, N_2 \rangle$ be $\Box$-spaces, and let $f \colon X_1 \to X_2$ be a $\Box$-morphism. Then the map
\[
\mathrm{U}(f) \colon \U(\X_2) \to \U(\X_1), \qquad \mathrm{U}(f)(V) := f^{-1}[V]
\]
is a lattice homomorphism preserving the operator $\Box$, that is,
\[
\mathrm{U}(f)(\Box_{\R_2}(V)) = \Box_{\R_1}(\mathrm{U}(f)(V)).
\]
\end{proposition}
\begin{proof}
We prove both inclusions. Let $x \in \Box_{\R_1}(\mathrm{U}(f)(V))$. Then $x \in N_1$ and $\R_1(x) \subseteq f^{-1}[V]$. By~\eqref{ax:box-morph-N}, it follows that $f(x) \in N_2$. Now let $y \in \R_2(f(x))$. By~\eqref{ax:box-morph2}, there exists $z \in X_1$ such that $x\R_1 z$ and $f(z) \leq y$. Since $z \in \R_1(x)$, we have $f(z) \in V$, and because $V$ is an up-set, it follows that $y \in V$. Therefore $\R_2(f(x)) \subseteq V$, and thus $f(x) \in \Box_{\R_2}(V)$, i.e., $x \in \mathrm{U}(f)(\Box_{\R_2}(V))$.

For the reverse inclusion, let $x \in \mathrm{U}(f)(\Box_{\R_2}(V))$. Then $f(x) \in N_2$ and $\R_2(f(x)) \subseteq V$. By~\eqref{ax:box-morph-N}, it follows that $x \in N_1$. Let $y \in \R_1(x)$. By~\eqref{ax:box-morph1}, $f(x) \R_2 f(y)$, hence $f(y) \in V$. Therefore $y \in f^{-1}[V]$, and so $\R_1(x) \subseteq f^{-1}[V]$. Thus $x \in \Box_{\R_1}(\mathrm{U}(f)(V))$.
\end{proof}

\begin{proposition}
\label{prop:Xh-box-morphism}
Let $\D_1$ and $\D_2$ be bounded distributive lattices with a modal operator $\Box$, and let $h \colon \D_1 \to \D_2$ be a lattice homomorphism preserving $\Box$. Then the map
\[
\mathrm{X}(h) \colon \mathrm{X}(\D_2) \to \mathrm{X}(\D_1), \qquad \mathrm{X}(h)(P) := h^{-1}[P]
\]
is a $\Box$-morphism.
\end{proposition}
\begin{proof}
We verify that $\mathrm{X}(h)$ satisfies the defining conditions of a $\Box$-morphism.

First, $\mathrm{X}(h)^{-1}[N_1] = N_2$ follows immediately from $h(\Box_1 1_1) = \Box_2 1_2$ and the definitions of $N_1$, $N_2$ and $\mathrm{X}(h)$.

Next, we prove that if $P\R_{\Box_2}Q$, then $\mathrm{X}(h)(P)\R_{\Box_1} \mathrm{X}(h)(Q)$, i.e., 
\[
\text{if }\Box_2^{-1}[P] \subseteq Q\text{, then }\Box_1^{-1}[h^{-1}[P]] \subseteq h^{-1}[Q].
\]
Let $x \in \Box_1^{-1}[h^{-1}[P]]$. Then $\Box_1 x \in h^{-1}[P]$, so $h(\Box_1 x) \in P$. Since $h$ preserves the modal operator, we have $h(\Box_1 x) = \Box_2(h(x))$, hence $\Box_2(h(x)) \in P$, and therefore $h(x) \in \Box_2^{-1}[P] \subseteq Q$. Thus $x \in h^{-1}[Q]$.

Finally, we verify condition~\eqref{ax:box-morph2}. Assume $P \in N_2$ and $\mathrm{X}(h)(P) \R_{\Box_1} Q$, that is, $\Box_1^{-1}[h^{-1}[P]] \subseteq Q$.
We prove that there exists $P' \in \mathrm{X}(\D_2)$ such that
\[
\Box_2^{-1}[P] \subseteq P' \quad \text{and} \quad h^{-1}[P'] \subseteq Q.
\]

Since $P \in N_2$, the set $\Box_2^{-1}[P]$ is a filter. From the hypothesis, if $h(\Box_1 a) \in P$, then $a \in Q$. Consider the set $h[Q^c]$ and observe that $h[Q^c] \cap \Box_2^{-1}[P] = \emptyset$. Indeed, if $b \in h[Q^c] \cap \Box_2^{-1}[P]$, then $\Box_2 b \in P$ and $b = h(a)$ for some $a \in Q^c$. Hence $\Box_2(h(a)) \in P$, so $h(\Box_1 a) \in P$, and by the hypothesis $a \in Q$, a contradiction.

Therefore, there exists a prime filter $P'$ such that
\[
\Box_2^{-1}[P] \subseteq P' \quad \text{and} \quad h[Q^c] \cap P' = \emptyset.
\]
It follows that $h^{-1}[h[Q^c]] \cap h^{-1}[P'] = \emptyset$. Since $Q^c \subseteq h^{-1}[h[Q^c]]$, we obtain $Q^c \cap h^{-1}[P'] = \emptyset$, and hence $h^{-1}[P'] \subseteq Q$. This completes the proof.
\end{proof}

The following theorem establishes that the representation on the topological side is faithful and recovers the relational structure.

\begin{theorem}
\label{thm:epsilon-box-morphism}
Let $\langle \X,\R,N\rangle$ be a $\Box$-space and consider the $\Box$-space given by 
\[
\langle \X(\U(\X)), \R_{\Box_{\R}}, \sigma_{\U(\X)}(\Box_{\R}(X))\rangle.
\] 
Then the map
\[
\epsilon \colon X \longrightarrow \mathrm{X}(\U(\X)), \qquad x \mapsto \epsilon(x) = \{V \in \U(\X) : x \in V\},
\]
is an isomorphism of $\Box$-spaces. In particular, $\epsilon$ is a homeomorphism, an order isomorphism and it satisfies
\[
\epsilon[N]=\sigma_{\U(\X)}(\Box_{\R}(X)),
\]
and 
\[
x \R y\iff \epsilon(x)\R_{\Box_{\R}} \epsilon(y).
\]
\end{theorem}
\begin{proof}
By Priestley duality, $\epsilon$ is already known to be an order-isomorphism and a homeomorphism. It therefore remains to verify that $\epsilon$ preserves the distinguished subset and the relation.

Since $\Box_{\R}(X)=N$, it suffices to prove that $\epsilon^{-1}[\sigma_{\U(\X)}(N)] = N$. Indeed, if $x \in \epsilon^{-1}[\sigma_{\U(\X)}(N)]$, then $\epsilon(x) \in \sigma_{\U(\X)}(N)$, which holds if and only if $N \in \epsilon(x)$, and this is equivalent to $x \in N$.

We now prove that $x \R y$ if and only if $\epsilon(x)\R_{\Box_{\R}} \epsilon(y)$.

Suppose that $\epsilon(x)\R_{\Box_{\R}} \epsilon(y)$ and that $x \notR y$. By condition~\ref{ax:boxR-epsilon} of Definition~\ref{def:box-space}, $x \in N$ and there exists $V \in \U(\X)$ such that $\R(x) \subseteq V$ and $y \notin V$. Then $x \in \Box_{\R}(V)$, and consequently $V \in \Box_{\R}^{-1}[\epsilon(x)] \subseteq \epsilon(y)$, which implies that $y \in V$, a contradiction. Therefore, $x\R y$.

Conversely, suppose that $x\R y$, and let $V \in \Box_{\R}^{-1}[\epsilon(x)]$. Then $x \in \Box_{\R}(V)$, which implies that $x \in N$ and $\R(x) \subseteq V$. Since $y \in \R(x)$, we obtain $y \in V$, and thus $V \in \epsilon(y)$. Hence, $\Box_{\R}^{-1}[\epsilon(x)] \subseteq \epsilon(y)$, which shows that $\epsilon(x)\R_{\Box_{\R}} \epsilon(y)$.
\end{proof}

The previous result shows that the relational structure of a $\Box$-space can be fully recovered from the operator $\Box_{\R}$ on clopen up-sets. We now establish the corresponding representation result on the algebraic side.

\begin{theorem}
\label{thm:sigma-box-iso}
Let $\D$ be a distributive lattice with a modal operator $\Box$ preserving finite meets. Then $\D$ is isomorphic to $\langle\U(\X(\D), \Box_{\R_\Box}\rangle$ via the map
\[
\sigma \colon \D \longrightarrow \U(\X(\D)), \qquad a \mapsto \sigma(a) = \{P \in \mathrm{X}(\D) : a \in P\}.
\]
\end{theorem}
\begin{proof}
The fact that $\sigma \colon D \to \U(\X(\D))$ is a lattice isomorphism is standard from Priestley duality. We have already shown in Proposition~\ref{prop:XD-box-space}(1) that $\sigma(\Box a) = \Box_{\R_\Box}(\sigma(a))$; thus $\sigma$ preserves the modal operators. Therefore, $\sigma$ is an isomorphism of modal lattices.
\end{proof}

As an immediate consequence of the previous theorems, we obtain the following duality result. 

\begin{corollary}
The category of bounded distributive lattices equipped with a finite meet-preserving $\Box$ operator and homomorphisms is dually equivalent to the category of $\Box$-spaces and $\Box$-morphisms.
\end{corollary} 

\subsection{Topological Duality for Modal Kleene triples}
\label{subsec:topological-duality-modal-kleene}

In this subsection, we extend the classical Priestley duality framework to the setting of modal Kleene algebras. To this end, we introduce the notion of a \emph{modal Kleene spaces}, which enriches the structure of dual Kleene triples with additional relations. We then establish a dual correspondence between modal Kleene triples and modal Kleene spaces, providing the basis for a topological duality theory.

\begin{definition}
\label{def:modal-dual-kleene-triple}
A \emph{modal Kleene space} is a tuple $\mathcal{X}=\langle \X,\G,\R,\Q,N,H\rangle$ such that:
\begin{itemize}
    \item $\langle \X, H, \Q \rangle$ is a dual Kleene triple;
    \item $\langle \X, \G \rangle$ is a $\Diamond$-space;
    \item $\langle \X, \R, N \rangle$ is a $\Box$-space.
\end{itemize}
In addition, the following conditions are required to hold for all $V, W \in \U(\X)$:
\begin{equation}
\label{ax:mdkt-filter}
\text{If } V \notcon_{\Q} W \text{ and } H \subseteq V \cup W, \text{ then } H \subseteq \Box_{\R} V \cup \Diamond_{\G} W,
\end{equation}
and, for all $x, y, y' \in X$,
\begin{equation}
\label{ax:mdkt-relation}
\text{If } x\Q y,\ x \in N \text{ and } y' \in \G(y), \text{ then } \exists x' \in \R(x) \text{ such that } x'\Q y'.
\end{equation}
\end{definition}

Intuitively, Condition~\eqref{ax:mdkt-filter} is the topological counterpart of the modal filter axiom~\eqref{cond:filter}, while Condition~\eqref{ax:mdkt-relation} is the topological counterpart of the modal contact axiom~\eqref{ax:modal-contact}, relating $\Q$, $\G$, and $\R$.

\begin{definition}
\label{def:mdkt-morphism}
Given modal Kleene spaces $\mathcal{X}_1=\langle \X_1,\G_1,\R_1,Q_1,N_1,H_1\rangle$ and $\mathcal{X}_2=\langle \X_2,\G_2,\R_2,\Q_2,N_2,H_2\rangle$, a function $f\colon X_1\to X_2$ is a \emph{modal dual Kleene morphism} if it is
\begin{itemize}
    \item a dual Kleene morphism;
    \item a $\Diamond$-morphism;
    \item a $\Box$-morphism.
\end{itemize}

The category of modal Kleene spaces and modal dual Kleene morphisms is denoted by $\mathbf{mKt}^{*}$.
\end{definition}

From the results established in Proposition~\ref{prop:Uf-box-morphism} and the results of~\cite{jalali, petrovich1996}, we obtain the following consequences:
\begin{itemize}
    \item If $f$ is a morphism of $\mathbf{mKt}^{*}$, then $\mathrm{U}(f)$ is a morphism of $\mathbf{mKt}$;
    \item If $h$ is a morphism of $\mathbf{mKt}$, then $\mathrm{X}(h)$ is a morphism of $\mathbf{mKt}^{*}$.
\end{itemize}

\begin{proposition}
\label{prop:triple-to-mdkt}
Let $\T = \langle \D, \con, F \rangle$ be a modal Kleene triple. Then
\[
\mathrm{X}(\T) = \langle \X(\D), \G_\Diamond, \R_\Box, \Q_{\con},\sigma(\Box 1), \mathcal{H}(F) \rangle
\]
is a modal Kleene space.
\end{proposition}
\begin{proof}
The fact that $\langle \X(\D), \mathcal{H}(F) ,\Q_{\con}\rangle$ is a dual Kleene triple was established in Proposition~\ref{prop:triple-to-dual} (see also~\cite{jalali}). Moreover, it is known from~\cite{petrovich1996} that $\langle \X(\D), \G_\Diamond \rangle$ is a $\Diamond$-space, and from Proposition~\ref{prop:XD-box-space} that $\langle \X(\D), \R_\Box, \sigma(\Box 1) \rangle$ is a $\Box$-space. It remains to verify conditions~\eqref{ax:mdkt-filter} and~\eqref{ax:mdkt-relation}.

We first verify condition~\eqref{ax:mdkt-filter}. Let $V, W$ be clopen up-sets. Then $V = \sigma(a)$ and $W = \sigma(b)$ for some $a, b \in D$. Suppose that $\sigma(a) \notcon_{\Q_{\con}} \sigma(b)$ and that $\sigma(a) \cup \sigma(b) \in \mathcal{F}(\mathcal{H}(F))$. We must show that
\[
\Box_{\R_\Box}(\sigma(a)) \cup \Diamond_{\G_\Diamond}(\sigma(b)) \in \mathcal{F}(\mathcal{H}(F)).
\]
By previous results, $\Box_{R_\Box}(\sigma(a)) = \sigma(\Box a)$ and $\Diamond_{G_\Diamond}(\sigma(b)) = \sigma(\Diamond b)$. Moreover, by Theorem~\ref{thm:sigma-iso}, the assumptions imply that $a \notcon b$ and $\mathcal{F}(\mathcal{H}(F)) = \sigma[F]$. Therefore, it suffices to prove that if $a \notcon b$ and $\sigma(a \vee b) \in \sigma[F]$, then $\sigma(\Box a \vee \Diamond b) \in \sigma[F]$, which is precisely axiom~\eqref{cond:filter} of modal Kleene triples.

We now verify condition~\eqref{ax:mdkt-relation}. Let $P_1, P_2, P_2' \in \mathrm{X}(\D)$ be such that
\[
P_1 \Q_{\con} P_2, \quad P_1 \in \sigma(\Box 1) \quad \text{and} \quad P_2' \in \G_\Diamond(P_2).
\]
We aim to show that there exists $P_1' \in \R_\Box(P_1)$ such that $P_1'\Q_{\con} P_2'$.

We claim that $\Box^{-1}[P_1] \times P_2' \subseteq \con$. Suppose otherwise. Then there exist $x \in \Box^{-1}[P_1]$ and $y \in P_2'$ such that $x \notcon y$. Since $\D$ is a modal contact lattice, it follows from~\eqref{ax:modal-contact} that $\Box x \notcon \Diamond y$. As $P_2' \subseteq \Diamond^{-1}[P_2]$, we have $\Diamond y \in P_2$, and since $x \in \Box^{-1}[P_1]$, we have $\Box x \in P_1$. Because $P_1 \Q_{\con} P_2$, it follows from~\eqref{eq:dual-contact} that $\Box x \con \Diamond y$, which contradicts $\Box x \notcon \Diamond y$. Therefore, $\Box^{-1}[P_1] \times P_2' \subseteq \con$.

Moreover, $\Box^{-1}[P_1]$ is a proper filter. Indeed, if $0 \in \Box^{-1}[P_1]$, it would follow that $0 \con 1$, which is impossible.

By Lemma~8 of~\cite{duntsch2008distributive}, applied to the proper filters $\Box^{-1}[P_1]$ and $P_2'$ under the condition $\Box^{-1}[P_1] \times P_2' \subseteq \con$, there exist $P_1', P_2'' \in \mathrm{X}(\D)$ such that $\Box^{-1}[P_1] \subseteq P_1'$, $P_2' \subseteq P_2''$, and $P_1' \times P_2'' \subseteq \con$. In particular, since $P_2' \subseteq P_2''$, it follows that $P_1' \times P_2' \subseteq \con$, and hence $P_1'\Q_{\con} P_2'$. Moreover, from $\Box^{-1}[P_1] \subseteq P_1'$ we obtain $P_1' \in \R_\Box(P_1)$. This completes the proof.
\end{proof}

We now establish the converse direction, showing that from any modal Kleene space we can recover a modal Kleene triple.

\begin{proposition}
\label{prop:mdkt-to-triple}
Let $\mathcal{X} = \langle \X, \G, \R, \Q,N, H \rangle$ be a modal Kleene space. Then $\mathrm{U}(\mathcal{X}) = \langle \U(\X), \con_{\Q}, \mathcal{F}(H) \rangle$ equipped with $\Box_{\R}$ and $\Diamond_{\G}$ is a modal Kleene triple.
\end{proposition}
\begin{proof}
It remains to verify Conditions~\eqref{ax:modal-contact} and~\eqref{cond:filter}. 

To prove Condition~\eqref{ax:modal-contact}, let $V, W$ be clopen up-sets. Suppose that $V \notcon_{\Q} W$. We will verify that $\Box_{\R} V \times \Diamond_{\G} W \subseteq \Q^c$. Suppose, towards a contradiction, that there exists $(x,y) \in \Box_{\R} V \times \Diamond_{\G} W$ such that $x \Q y$. From $(x,y) \in \Box_{\R} V \times \Diamond_{\G} W$ it follows that $x \in N$, $\R(x) \subseteq V$, and there exists $y' \in W \cap \G(y)$. By condition~\eqref{ax:mdkt-relation}, there exists $x' \in \R(x)$ such that $x'\Q y'$. Since $\R(x) \subseteq V$, we have $x' \in V$, and thus $(x', y') \in V \times W$, contradicting the hypothesis that $V \times W \subseteq \Q^c$. Therefore $\Box_{\R} V \notcon_{\Q} \Diamond_{\G} W$, as required. 

Condition~\eqref{cond:filter} follows directly from the fact that modal dual Kleene triples satisfy Condition~\eqref{ax:mdkt-filter}.
\end{proof}

\begin{theorem}
\label{thm:epsilon-mdkt-morphism}
Let $\mathcal{X} = \langle \X, \G, \R, \Q,N, H \rangle$ be a modal Kleene space. Then the map
\[
\epsilon \colon \X \longrightarrow \X(\U(\X)), \qquad x \mapsto \epsilon(x) = \{V \in \U(\X) : x \in V\},
\]
is an isomorphism of modal Kleene spaces. In particular, $\X$ and $\X(\U(\X))$ are isomorphic as ordered sets and homeomorphic as topological spaces. Moreover, for all $x, y \in X$ we have
\begin{align*}
x \Q y &\iff \epsilon(x)\Q_{\con_{\Q}} \epsilon(y),\\
x\G y &\iff  \epsilon(x)\G_{\Diamond_{\G}} \epsilon(y),\\
x\R y &\iff \epsilon(x)\R_{\Box_{\R}} \epsilon(y),
\end{align*}
and $\epsilon[H] = \mathcal{H}(\mathcal{F}(H))$.
\end{theorem}
\begin{proof}
The result follows immediately from Theorem~\ref{thm:epsilon-iso}, which establishes that $\epsilon$ is a Priestley isomorphism preserving the relation $\Q$ and the distinguished set $H$, combined with Theorem~\ref{thm:epsilon-box-morphism}, which handles the preservation of the relation $\R$ and the distinguished set $N$. The remaining preservation properties follow from the analogous representation result for $\Diamond$-spaces established in~\cite{petrovich1996}.
\end{proof}

\begin{theorem}
\label{thm:sigma-mdkt-iso}
Let $\T = \langle \D, \con, F \rangle$ be a modal Kleene triple. Then the map
\[
\sigma \colon \D \longrightarrow \U(\mathrm{X}(\D)), \qquad a \mapsto \sigma(a) = \{P \in \mathrm{X}(\D) : a \in P\},
\]
is an isomorphism of modal lattices. Moreover, for all $a, b \in D$,
\[
a \con b \iff \sigma(a) \con_{\Q_{\con}} \sigma(b),
\]
and $\sigma[F] = \mathcal{F}(\mathcal{H}(F))$.
\end{theorem}
\begin{proof}
This is established by combining Theorem~\ref{thm:sigma-iso}, the analogous algebraic representation for the operator $\Diamond$ established in~\cite{petrovich1996}, and Theorem~\ref{thm:sigma-box-iso}. Together, these results ensure that $\sigma$ is a modal lattice isomorphism that preserves the contact relation $\con$ and both modal operators.
\end{proof}

As a consequence of the preceding results, we obtain that $\mathrm{U}$ and $\mathrm{X}$ also extend naturally to the modal setting.

\begin{theorem}
\label{thm:mKt-duality}
The functors $\mathrm{U} \colon \mathbf{mKt}^{*} \to \mathbf{mKt}$ and $\mathrm{X} \colon \mathbf{mKt} \to \mathbf{mKt}^{*}$ establish a dual equivalence between the categories $\mathbf{mKt}$ and $\mathbf{mKt}^{*}$.
\end{theorem}

By composing the categorical equivalence established above with the previous duality, we obtain a duality for modal Kleene algebras.

\begin{corollary}
    The functors $\mathrm{K} \circ \mathrm{U}\colon \mathbf{mKt}^{*}\to \mathbf{mKl}$ and $ \mathrm{X} \circ \mathrm{T}\colon \mathbf{mKl}\to \mathbf{mKt}^{*}$  establish a dual equivalence between the categories $\mathbf{mKl}$ and $\mathbf{mKt}^{*}$.
\end{corollary}

\subsection*{Acknowledgments}
This research was funded by the Argentinean projects PIP 112-20200101301CO (CONICET) and 112-20250100318CO (CONICET). Finally, it acknowledges the support of the Argentinean project 03/C338 (UNICEN).


\end{document}